\documentclass[11pt]{article}%{amsart}%{article}
\usepackage{amsmath}%
\usepackage{amsfonts}%
\usepackage{amssymb}%
\usepackage{graphicx}
\usepackage{mathtools}
\usepackage[T2A]{fontenc}
\usepackage[cp1251]{inputenc}
\usepackage{subcaption}

\usepackage{xcolor}

\usepackage{amsthm} 
\usepackage[colorlinks]{hyperref}

\AtBeginDocument{\hypersetup{
    citecolor=magenta,
    colorlinks=true,
    linkcolor=blue,
    filecolor=magenta,      
    urlcolor=cyan,    
    pdfpagemode=FullScreen,
    }}

\usepackage{enumerate}

\usepackage[all,cmtip]{xy}

\usepackage{doi}

\DeclareMathOperator{\id}{id}

\DeclareMathOperator{\GL}{GL}
\DeclareMathOperator{\SL}{SL}
\DeclareMathOperator{\Ad}{Ad}

\DeclareMathOperator{\Adj}{Adj}
\DeclareMathOperator{\Spin}{Spin}

\newcommand{\pprime}{{\prime\prime}}

\newtheorem{theorem}{Theorem}[section]
\theoremstyle{definition}
\newtheorem{definition}[theorem]{Definition}
\newtheorem{problem}[theorem]{Problem}
\newtheorem{example}[theorem]{Example}
\newtheorem{lemma}[theorem]{Lemma}
\newtheorem{proposition}[theorem]{Proposition}
\theoremstyle{definition}\newtheorem{remark}[theorem]{Remark}
\newtheorem{corollary}[theorem]{Corollary}
\theoremstyle{remark}
\numberwithin{equation}{section}

\title{Multiple rational normal forms in Lie theory}
\author{Dmitriy Voloshyn}
\date{}

\newcounter{myfootnote}
\begin{document}
\maketitle

\makeatletter
\def\blfootnote{\xdef\@thefnmark{}\@footnotetext}
\makeatother
\begin{abstract}
    We study the decomposition of a generic element $g \in G$ of a connected reductive complex algebraic group $G$ in the form $g = N(g) B(g) \bar{u} N(g)^{-1}$ where $N: G \dashrightarrow \mathcal{N}_-$ and $B : G \dashrightarrow \mathcal{B}_+$ are rational maps onto a unipotent subgroup $\mathcal{N}_-$ and a Borel subgroup $\mathcal{B}_+$ opposite to $\mathcal{N}_-$, and $\bar{u}$ is a representative of a Weyl group element $u$. We introduce a class of rational Weyl group elements that give rise to such decompositions, and study their various properties.
\end{abstract}

\tableofcontents

%20G07 Structure theory of linear algebraic groups
%20F55 Coxeter groups and root systems
%13F60 Cluster algebras

\section{Introduction}
 Let $G$ be a connected complex reductive algebraic group, $(\mathcal{B}_+,\mathcal{B}_-)$ a pair of opposite Borel subgroups, $\mathcal{N}_\pm \subset \mathcal{B}_\pm$ their unipotent radicals, $\mathcal{H}:=\mathcal{B}_+\cap \mathcal{B}_-$ a Cartan subgroup, $W:=N_G(\mathcal{H})/\mathcal{H}$ the corresponding Weyl group. The goal of the paper is to address the following problem:

\begin{problem}\label{problem}
    For a given $G$, describe elements $u$ of its Weyl group $W$ for which there exist rational maps $N : G \dashrightarrow \mathcal{N}_-$ and $B : G \dashrightarrow \mathcal{B}_+$ such that for a generic $g \in G$, 
    \begin{equation}\label{eq:problem_dec}
        g = N(g) B(g)\overline{u}N(g)^{-1}
    \end{equation} 
    where $\overline{u} \in N_G(\mathcal{H})$ is a representative of $u$.
\end{problem}

Throughout the paper, we call an element $u \in W$ a \emph{solution} of Problem~\ref{problem} if there exist rational maps $B$ and $N$ that give rise to the decomposition~\eqref{eq:problem_dec}. In the next example, we illustrate the problem in $\GL_2(\mathbb{C})$.

\begin{example}
    For $G := \GL_2(\mathbb{C})$, identify the Weyl group with the group of permutation matrices:
    \begin{equation*}
        W = \left \{ w_0:=\begin{bmatrix}
            0 & 1 \\ 1 & 0
        \end{bmatrix}, \ I:=\begin{bmatrix}
1 & 0\\ 
0 & 1
\end{bmatrix}\right\}.
    \end{equation*}
    Set $D:= \{X \in \GL_2(\mathbb{C})\ | \ x_{12} \neq 0\}$. Then every $X \in D$ has a decomposition of the form
    \begin{equation*}
        X = \begin{bmatrix}
            1 & 0\\
            \frac{x_{22}}{x_{12}} & 1 
        \end{bmatrix} \begin{bmatrix}
            x_{12} & x_{11}+x_{22}\\
            0 & -\frac{\det X}{x_{12}} 
        \end{bmatrix} \begin{bmatrix} 0 & 1\\ 1 & 0
        \end{bmatrix}\begin{bmatrix}
            1 & 0\\
            -\frac{x_{22}}{x_{12}} & 1 
        \end{bmatrix};
    \end{equation*}
    therefore, the element $w_0$ is a solution of Problem~\ref{problem}. On the other hand, let us set
    \begin{equation*}
        \Delta:=\Delta(X):=x_{11}^2 - 2x_{11}x_{12} + x_{22}^2 + 4x_{12}x_{21}.
    \end{equation*}
    Then, for the identity permutation, $X \in D$ decomposes as
    \begin{equation*}%\label{eq:example2irr}
        X = \begin{bmatrix}
            1 & 0\\
            \frac{-x_{11}\pm \sqrt{\Delta}}{2x_{12}} & 1 
        \end{bmatrix} 
        \begin{bmatrix}
            \frac{1}{2}(x_{11}+x_{22}) \pm \frac{1}{2}\sqrt{\Delta} & x_{12}\\ 0 & \frac{1}{2}(x_{11}+x_{22}) \mp \frac{1}{2}\sqrt{\Delta}
        \end{bmatrix}
        \begin{bmatrix}
            1 & 0\\
            -\frac{-x_{11}\pm \sqrt{\Delta}}{2x_{12}} & 1 
        \end{bmatrix} 
    \end{equation*}
    Clearly, the identity permutation is not a solution of Problem~\ref{problem}.
\end{example}

Our motivation for studying Problem~\ref{problem} originates in the theory of cluster algebras, as defined by S. Fomin and A. Zelevinsky in~\cite{fathers}, and in particular, in its interaction with Poisson geometry. For an overview of the subject and its applications, we refer to~\cite{fomin13,fomin_intro,dasbuch,leclerc,williams_intro}.

In~\cite{double}, M. Gekhtman, M. Shapiro, and A. Vainshtein constructed a generalized cluster structure in the ring of invariants\footnote{We use the notation $\Ad\mathcal{N}_-$ to distinguish the invariant ring from the one where $\mathcal{N}_-$ acts by left or right translations. So, $p \in \mathbb{C}[\GL_n]^{\Ad\mathcal{N}_-}$ if and only if $p(ngn^{-1}) = p(g)$ for all $g\in G$, $n \in \mathcal{N}_-$. } $\mathbb{C}[\GL_n]^{\Ad \mathcal{N}-}$, where $\mathcal{N}_-$ acts on $\GL_n$ by conjugation. The initial extended cluster of $\mathbb{C}[\GL_n]^{\Ad \mathcal{N}_-}$ is a part of the initial extended clusters in $\mathbb{C}[\GL_n]$ and $\mathbb{C}[\GL_n \times \GL_n]$. In the first case, the obtained cluster structure is compatible\footnote{A cluster structure is called \emph{compatible} with a Poisson bracket if in any given extended cluster, the Poisson bracket is quadratic.} with the Poisson bracket induced from the Poisson dual $\GL_n^*$ of $\GL_n$, and in the second case, the Poisson structure on $\GL_n \times \GL_n$ is that of the Drinfeld double of $\GL_n$. As we showed in \cite{multdual,multdouble}, the same initial extended cluster of $\mathbb{C}[\GL_n]^{\Ad \mathcal{N}_-}$ is a part of the initial extended clusters in $\mathbb{C}[\GL_n]$ and $\mathbb{C}[\GL_n \times \GL_n]$ compatible with the Poisson brackets from the Belavin--Drinfeld class. We refer to \cite{chari,etingof} for details on Poisson geometry.

The construction of the initial extended cluster in $\mathbb{C}[\GL_n]^{\Ad\mathcal{N}_-}$ utilizes a solution of Problem~\ref{problem}. Let $C:=s_1 \cdots s_{n-1}$ be a special Coxeter element in type $A_{n-1}$. By \cite[Lemma 8.2]{double}, there exist rational maps $B : \GL_n(\mathbb{C}) \dashrightarrow \mathcal{B}_+$ and $N : \GL_n(\mathbb{C}) \dashrightarrow \mathcal{N}_-$ such that for a generic $X \in \GL_n(\mathbb{C})$,
\begin{equation*}
    X = N(X)B(X)\overline{C}N(X)^{-1}.
\end{equation*}
The initial extended cluster comprises $n-1$ isolated frozen variables, which are the characters $\chi^{\omega_i}$ corresponding to the fundamental weights $\omega_1,\ldots,\omega_{n-1}$,
and the rest of the variables is obtained from the initial extended cluster of $\mathbb{C}[\GL_{n-1}]^{\mathcal{N}_-}$, where $\mathcal{N}_-$ acts by multiplication on the right, via the pullback by the map
\begin{equation*}
    \GL_n \ni X \mapsto B(X)^{[2,n]}_{[2,n]} \in \GL_{n-1}.
\end{equation*}

It is not known how to extend the above construction to other Lie types. In order to construct a generalized cluster structure in $\mathbb{C}[G]^{\Ad \mathcal{N}_-}$ (and in related Poisson varieties), we are thus led to the study of Problem~\ref{problem}, which we also find to be interesting in its own right.

Let us briefly explain our results. We call an element $u \in W$ of a Weyl group $W$ \emph{rational} if a certain graph $\Gamma(u)$ associated with $u$ is acyclic (see Definition~\ref{d:rational}). In \cite[Lemma 8.3]{double}, it was shown that the longest Weyl group element $w_0 \in W$ is a solution of Problem~\ref{problem} in any Lie type. We develop a recursive procedure that constructs a pair of maps $(B,N)$ for $u$ from the pair of maps for $w_0$ (or, more generally, from any other solution of Problem~\ref{problem}). If $u$ is rational, then the recursive procedure ends in a finite number of steps, and thus one obtains the decomposition~\eqref{eq:problem_dec} for $u$. We also prove a partial converse: if $u$ is a sufficiently long element and is a solution of Problem~\ref{problem}, then $u$ is rational (see Theorem~\ref{thm1}). We conjecture that the class of solutions of Problem~\ref{problem} is precisely the class of rational Weyl group elements.

We then define a \emph{rationality graph} $\Gamma(W)$ whose vertices are the rational Weyl group elements of $W$, and two vertices $u,v \in W$ are connected by an edge iff there is a simple reflection $s$ such that $u = sv$. Surprisingly, the graph $\Gamma(W)$ is connected, and it contains more than one vertex if and only if $w_0$ induces a nontrivial automorphism of the Dynkin diagram; that is, only in type $A_r$, $r \geq 2$, $D_r$ for $r$ odd, and $E_6$ (see Theorem \ref{thm2}; for examples of $\Gamma(W)$, see Appendix~\ref{appendix}).

In the context of constructing a generalized cluster structure in $\mathbb{C}[\SL_n]^{\Ad\mathcal{N}_-}$, the special Coxeter element $C$ and its inverse $C^{-1}$ exhibit the following properties: 1) The two elements are the only rational Coxeter elements; 2) The two elements are the only vertices of valency $1$ in $\Gamma(W)$. However, no Coxeter element in any other Lie type is rational; nevertheless, in type $D_r$ for $r$ odd, the rationality graph $\Gamma(W)$ also contains two elements of valency $1$, and this similarity leads us to believe that the two elements play a role in the construction of a generalized cluster structure in $\mathbb{C}[\Spin(2r)]^{\Ad \mathcal{N}_-}$. We will study this question in our next paper.

We give precise definitions and statements of our main results in the next section. In Sections~\ref{s:proofs1} and \ref{s:proofs2} we provide the proofs.

\paragraph{Acknowledgements.} At various stages of this work, we benefited from discussions with Benjamin Enriquez, Michael Gekhtman, Antoine de Saint Germain, Igor Krylov, Jiang-Hua Lu, Hugh Thomas and Anderson Vera. This research was supported by the Institute for Basic Science (IBS-R003-D1).
 \section{Main results}\label{s:main_results}
 Let us fix the setup and notation for the rest of the work. Let $G$ be a connected complex reductive algebraic group and $\mathfrak{g}$ be its Lie algebra. Denote by $\Pi$ the set of roots of $\mathfrak{g}$, $\Delta$ a set of simple roots, $\Pi_+$ and $\Pi_-$ the sets of positive and negative roots (in Section~\ref{s:main_results2} and Section~\ref{s:proofs2}, we further assume that the root system $\Pi$ is indecomposable). We let $\mathcal{B}_+$ and $\mathcal{B}_-$ to be the corresponding Borel subgroups of $G$, $\mathcal{N}_\pm \subseteq \mathcal{B}_\pm$ their unipotent radicals, $\mathcal{H} :=\mathcal{B}_+ \cap \mathcal{B}_-$ the corresponding Cartan subgroup, $W:=N_G(\mathcal{H})/\mathcal{H}$ the Weyl group.

\subsection{Solutions of Problem \ref{problem}}\label{s:main_results1}
In this subsection, we state the first part of our main results. The proofs are collected in Section~\ref{s:proofs1}.

For a subset $A \subseteq \Pi_+$, denote 
\begin{equation}
    \Adj (A):= \{ \alpha \in \Pi_+ \ | \ \exists \beta \in A : \ \alpha \leq \beta\}.
\end{equation}
\begin{definition}\label{d:nuseq}
The \emph{$\nu$-sequence} of an element $u \in W$ is the sequence $\{\nu^k(u)\}_{k \geq 0}$ of subsets of $\Pi_+$ defined as
\begin{equation}\label{eq:nuseq}
    \nu^0(u):=u(\Pi_+) \cap \Pi_+, \ \ \nu^k(u) := u(\Adj \nu^{k-1}(u))\cap \Pi_+, \ \ k \geq 1.
\end{equation}
\end{definition}
In Proposition~\ref{p:nu_descends}, we show that the sequence $\{\nu^k(u)\}_{k \geq 0}$ is descending and stabilizes after a finite number of terms. Thus we can make sense of the limit
\begin{equation}
    \nu(u) := \lim_{k \rightarrow \infty} \nu^k(u);
\end{equation}
that is, if $\nu^m(u) = \nu^{m+1}(u) = \cdots$ for some $m \geq 1$, then $\nu(u) = \nu^m(u)$.

\begin{example}\label{e:nuseq1}
    Consider the Coxeter element $C:= s_1 s_2$ in type $A_2$. The associated $\nu$-sequence is given by $\nu^0(C) = \{\alpha_2\}$, $\nu(C) = \nu^1(C) = \emptyset$. 
\end{example}
\begin{example}\label{e:nuseq2}
    Consider the Coxeter element $C:=s_1 s_2$ in type $B_2$, with $\alpha_2$ being the short root. Then $\nu^0(C) = \{\alpha_2,\alpha_1+2\alpha_2\}$. Since $\theta:=\alpha_1 + 2\alpha_2$ is the highest root in type $B_2$, we see that $\Adj \nu^0(C) = \Pi_+$, and therefore, $\nu(C) = \nu^0(C) \neq \emptyset$.
\end{example}
For more examples of $\nu$-sequences, see Sections~\ref{s:rat_coxeter}-\ref{s:typeDrvalency1}.

\begin{definition}\label{d:gammau}
    The \emph{rationality graph} $\Gamma(u)$ of an element $u \in W$ is an oriented graph constructed as follows:
    \begin{itemize}
        \item The vertex set of $\Gamma(u)$ is given by the set $\nu^0(u)$;
        \item Two roots $\alpha,\beta \in \nu^0(u)$ are connected by an edge $\alpha \rightarrow \beta$ if and only if $u^{-1}(\alpha) \leq \beta$.
    \end{itemize}
\end{definition}

\begin{example}
    For the Coxeter element $s_1s_2$ in type $A_2$, the rationality graph consists of a single vertex $\alpha_2$ with no loop. For the Coxeter element $s_1s_2$ in type $B_2$ from Example~\ref{e:nuseq2}, the rationality graph is given by
    \begin{equation*}
        \xymatrix{
        \alpha_2 \ar[r] & \theta \ar@(ru,dr)
        }
    \end{equation*}
\end{example}

\begin{definition}\label{d:rational}
    An element $u \in W$ is called \emph{rational} if one of the following equivalent conditions is satisfied:
    \begin{enumerate}[i)]
        \item The set $\nu(u)$ is empty;\label{i:rational1}
        \item The graph $\Gamma(u)$ has no cycles.\label{i:rational2}
    \end{enumerate}
\end{definition}
We prove the equivalence of the above conditions in Proposition~\ref{p:nu_nonrat}. In Proposition~\ref{p:fixedpt} we show that $\Gamma(u)$ contains a loop if and only if $\alpha \leq u(\alpha)$ for some $\alpha \in \Pi_+$. As a consequence, a rational element does not fix any root. Another consequence is that if the highest root is a vertex of $\Gamma(u)$, then $u$ is not rational (see Proposition~\ref{p:highrootnonrat}).

\begin{theorem}\label{thm1}
    For any rational element $u \in W$ and its representative $\overline{u} \in N_G(\mathcal{H})$, there exist rational maps $B : G \dashrightarrow \mathcal{B}_+$ and $N : G \dashrightarrow \mathcal{N}_-$ such that for a generic $g \in G$
    \begin{equation}\label{eq:thm1dec}
        g = N(g) B(g) \overline{u} N(g)^{-1}.
    \end{equation}
    Conversely, if $u \in W$ is any element with $\ell(u) \geq \ell(w_0) - 2$ and the decomposition~\eqref{eq:thm1dec} exists for a generic $g \in G$ and some rational maps $B$ and $N$, then $u$ is rational.
\end{theorem}

In other words, any rational element $u \in W$ is a solution of Problem~\ref{problem}, and conversely, any solution $u$ of Problem~\ref{problem} with $\ell(u) \geq \ell(w_0) - 2$ is rational. We conjecture that the solutions of Problem~\ref{problem} are precisely the rational Weyl group elements. 

The main part of Theorem~\ref{thm1} rests on the following construction, which we precede with a technical definition.

\begin{definition}\label{d:nu_relative}
    Given $u,v \in W$, the \emph{$\nu$-sequence of $u$ relative to $v$} is the sequence $\{\nu^k(u\,|\,v)\}_{k \geq 0}$ of subsets of $\Pi_+$ defined as follows:
    \begin{equation}
    \begin{aligned}
        &\nu^0(u\,|\,v) := uv^{-1}(\Pi_-) \cap \Pi_+;\\
        &\nu^k(u\,|\,v) := u(\Adj \nu^{k-1}(u\,|\,v)) \cap \Pi_+.
    \end{aligned}
    \end{equation}
\end{definition}
If the $\nu$-sequence of $u$ relative to some $v$ stabilizes after a finitely many terms, we set
\begin{equation}
    \nu(u\,|\,v) := \lim_{k \rightarrow \infty} \nu^k(u\,|\,v).
\end{equation}
Clearly, the $\nu$-sequence of $u$ (see Definition~\ref{d:nuseq}) is the $\nu$-sequence of $u$ relative to $w_0$.

For a generic element $g \in G$, denote by $g_- \in \mathcal{N}_-$ and $g_{\oplus} \in \mathcal{B}_+$ the unique elements such that $g = g_- g_\oplus$ (the Gaussian decomposition). Let us a fix any solution $v \in W$ of Problem~\ref{problem}, and let us denote by $\tilde{B} : G \dashrightarrow \mathcal{B}_+$ and $\tilde{N} : G \dashrightarrow \mathcal{N}_-$ the rational maps such that for a generic $g \in G$,
\begin{equation}\label{eq:decv}
    g = \tilde{N}(g) \tilde{B}(g) \overline{v} \tilde{N}(g)^{-1}
\end{equation}
for some representative $\overline{v}$ of $v$. For any other element $u \in W$, construct a sequence of rational maps $\{P_k : G \dashrightarrow G\}_{k \geq 0}$ via
\begin{equation}\label{eq:pk}
\begin{aligned}
    &P_0(g):= \tilde{B}(g)\overline{v}\,\overline{u}^{-1},\\
    &P_k(g) := [P_{k-1}(g)]_{\oplus} \overline{u} [P_{k-1}(g)]_- \overline{u}^{-1}, \ \ k \geq 1;
\end{aligned}
\end{equation}
and a sequence of rational maps $\{N_k : G \dashrightarrow \mathcal{N}_-\}_{k \geq 0}$ as
\begin{equation}\label{eq:nk}
    N_k(g) := \tilde{N}(g) \tilde{N}(P_k(g)\overline{u})^{-1}, k \geq 0. 
\end{equation}
\begin{proposition}\label{p:relative_existence}
    For any solution $v \in W$ of Problem~\ref{problem} and any $u \in W$ such that $\nu(u \, | \, v) = \emptyset$, $u$ is also a solution of Problem~\ref{problem}. More precisely, for a given representative $\overline{u}$, there is $m \geq 0$ such that for a generic $g \in G$,
    \begin{align*}
        &B(g) := P_m(g) = P_{m+1}(g) = \ldots \in \mathcal{B}_+;\\
        &N(g) := N_m(g) = N_{m+1}(g) = \ldots \in \mathcal{N}_-,
    \end{align*}
    and such that
    \begin{equation*}
        g = N(g) B(g) \overline{u} N(g)^{-1}.
    \end{equation*}
\end{proposition}
It was shown in \cite[Lemma 8.3]{double} that the longest Weyl group element $w_0$ is a solution of Problem~\ref{problem}. More precisely, let us set
\begin{align}
    &N_0(g) = [g\overline{w_0}^{-1}]_-;\label{eq:w0n0} \\
    &B_0(g) = [g\overline{w_0}^{-1}]_{\oplus} \overline{w_0} ([g\overline{w_0}^{-1}]_-)\overline{w_0}^{-1}.\label{eq:w0b0}
\end{align}
It is straightforward to verify that for a generic $g \in G$,
\begin{equation}
    g = N_0(g) B_0(g) \overline{w_0} N_0(g)^{-1}.
\end{equation}
One direction of Theorem~\ref{thm1} is thus obtained as a corollary of Proposition~\ref{p:relative_existence} by setting $v:= w_0$. 

To test whether an element $u \in W$ is not a solution of Problem~\ref{problem}, we have the following methods: 1) The contrapositive to Proposition~\ref{p:relative_existence} (see Corollary~\ref{c:contrapositive}; a more practical statement is Corollary~\ref{c:negative_notsolution}); 2) Determining whether the generic fiber of the restriction $B_0|_{\mathcal{B}_+\overline{u}} : \mathcal{B}_+\overline{u} \rightarrow \mathcal{B}_+$ is of cardinality $> 1$ (see Lemma~\ref{l:fibers}). The statements needed for proving the partial converse to Theorem~\ref{thm2} can be summarized as follows.

\begin{proposition}\label{p:simple_nonexist}
    Let $u \in W$ be any element and $\alpha \in \Delta$ be a simple root.
    \begin{enumerate}[1)]
        \item If $u$ is not a solution of Problem~\ref{problem} and $u(\alpha) \in \Pi_-$, then $s_{\alpha}u$ is not a solution of Problem~\ref{problem}.\label{i:simple_nonexist1}
        \item If $\Gamma(u)$ has a cycle consisting of pairwise orthogonal simple roots, then $u$ is not a solution of Problem~\ref{problem}. In particular, if $u$ fixes a simple root, it is not a solution.\label{i:simple_nonexist2}
        \item Assume there exist $\alpha_i,\alpha_j \in \Delta$ such that $a_{ij} = a_{ji} = -1$ (Cartan integers), with  $\alpha_i,\alpha_i+\alpha_j \in \nu^0(u)$, and $\Gamma(u)$ contains a $2$-cycle on $\alpha_i$ and $\alpha_i+\alpha_j$. Then $u$ is not a solution of Problem~\ref{problem}.\label{i:simple_nonexist3}
    \end{enumerate}
\end{proposition}

The next results concern general properties of the maps $B$ and $N$. For a given $u \in W$ and its representative $\overline{u}$, define a rational map $\mathcal{T}_{\overline{u}} : \mathcal{N}_- \times \mathcal{B}_+ \rightarrow G$ via
\begin{equation}\label{eq:nbrat}
        \mathcal{T}_{\overline{u}}(n,b) = nb\overline{u}n^{-1}, \ \ n \in \mathcal{N}_-, \ b \in \mathcal{B}_+,
    \end{equation}
\begin{proposition}\label{p:nbrat}
    An element $u \in W$ is a solution of Problem~\ref{problem} if and only if the regular map $\mathcal{T}_{\overline{u}}$ is birational.
\end{proposition}
%The next proposition is essentially a corollary of Proposition~\ref{p:nbrat}.
\begin{proposition}\label{p:general}
    Let $u \in W$ be a solution of Problem~\ref{problem}. Then the corresponding rational maps $B : G \dashrightarrow \mathcal{B}_+$ and $N : G \dashrightarrow \mathcal{N}_-$ satisfy the following properties:
    \begin{enumerate}[1)]
        \item Both maps $B$ and $N$ are dominant and unique;
        \item For a generic $g \in G$ and $n \in \mathcal{N}_-$,
        \begin{align}
            &B(ngn^{-1}) = B(g);\label{eq:binvar} \\
            &N(ngn^{-1}) = nN(g).
        \end{align}
        \item The rational map $N$ is determined by $B$ via the equation
        \begin{equation}\label{eq:nfromb}
            N(g) = N_0(g) N_0(B(g)\overline{u})^{-1} = [g\overline{w_0}^{-1}]_- ([B(g)\overline{u}\,\overline{w_0}^{-1}]_-)^{-1}.
        \end{equation}
        for any representative $\overline{w_0}$ of $w_0$.
    \end{enumerate}
\end{proposition}
In particular, in the context of Proposition~\ref{p:relative_existence}, if the sequence $\{N_k : G \dashrightarrow \mathcal{N}_-\}$ given by~\eqref{eq:nk} stabilizes, its limit is given by~\eqref{eq:nfromb}.

\subsection{Combinatorics of rational Weyl group elements}\label{s:main_results2}
In this subsection, we state the second part of our main results, which concern combinatorial properties of rational Weyl group elements. The proofs are collected in Section~\ref{s:proofs2}. Here, we assume that the root system $\Pi$ is indecomposable.

\begin{definition}
    The \emph{rationality graph} $\Gamma(W)$ of $W$ is the graph constructed as follows:
    %\begin{enumerate}[1)]
    \begin{itemize}
        \item The vertices of $\Gamma(W)$ are the rational Weyl group elements;
        \item Two vertices $u$ and $v$ are connected by an edge if and only if there exists a simple reflection $s_{\alpha}$, $\alpha \in \Delta$, such that $u = s_{\alpha}v$.
    \end{itemize}
   % \end{enumerate}
\end{definition}
If we know that $W$ is of some specific Lie type (say, $A_r$), then we write $\Gamma(A_r)$ instead of $\Gamma(W)$. Some examples of rationality graphs are given in Figure~\ref{f:ratgraphs}.

\begin{theorem}\label{thm2}
    The rationality graph $\Gamma(W)$ is connected and carries a $\mathbb{Z}_2$-symmetry centered at the vertex $w_0$. Moreover, the graph $\Gamma(W)$ has more than one vertex if and only if $w_0 \neq -\id$.
\end{theorem}
In other words, there exist rational elements $u \in W$ other than $w_0$ if and only if the Lie type of $W$ is either $A_r$ for $r \geq 2$, or $D_r$ for $r$ odd, or $E_6$. Our proof is constructive: for every rational $u \in W$, we find an explicit path to $w_0$.

For type $A_r$, define the following special Coxeter element:
\begin{equation}\label{eq:specialcoxeter}
    C:= s_1 s_2 \cdots s_r.
\end{equation}

\begin{proposition}\label{p:rat_coxeter}
    A Coxeter element $c \in W$ is rational if and only if the Lie type of $W$ is $A_r$, $r \geq 1$, and $c \in \{C,C^{-1}\}$. In this case, $c$ has valency $1$ in $\Gamma(A_r)$.
\end{proposition}

As we explained in the introduction, the special Coxeter element~\eqref{eq:specialcoxeter} plays a role in the construction of a generalized cluster structure in the ring of invariants $\mathbb{C}[\SL_n]^{\Ad\mathcal{N}_-}$ (that is, regular functions invariant under the action $(n,g) \mapsto ngn^{-1}$, $n \in \mathcal{N}_-$, $g \in G$). Although Coxeter elements are not rational in type $D_r$ for $r$ odd, there is at least the following similarity:
\begin{proposition}\label{p:typeDrvalency1}
    In type $D_r$ for $r \geq 5$ odd, the rationality graph $\Gamma(D_r)$ contains at least $2$ vertices of valency $1$.
\end{proposition}

We give an explicit construction of the two elements of valency $1$ in type $D_r$ (see Proposition~\ref{p:specialDr}). In type $E_6$, however, there are no rational elements of valency $1$; instead, there are four elements of valency $2$. We expect that both Proposition~\ref{p:rat_coxeter} and Proposition~\ref{p:typeDrvalency1} can be strengthened: the rationality graphs $\Gamma(A_r)$, $r \geq 2$, and $\Gamma(D_r)$ for $r$ odd have exactly $2$ vertices of valency $1$.

If our main conjecture holds, the neighborhood of a given vertex $u \in W$ in $\Gamma(W)$ consists of those $s_{\alpha}u$ for which $\nu(s_{\alpha}u\,|\,u) = \emptyset$. In other words, if $u$ is rational, then $s_{\alpha}u$ is rational if and only if $\nu(s_{\alpha}u\,|\,u) = \emptyset$. %We prove the latter statement only in the case of $u(\alpha) < 0$ (see Lemma (...)).

In Tables~\ref{t:countA} and \ref{t:countDE}, we present the numbers of rational Weyl group elements in types with nontrivial $w_0$. We do not know if there is a closed formula for type $A$; in type $D_r$ for $r$ odd, we conjecture that the number of rational Weyl group elements is equal to $2^r -1$.

\begin{table}[ht]
\centering
\begin{tabular}{|l|l|l|l|l|l|l|l|l|l|l|l|}
\hline
Type  & $A_1$ & $A_2$ & $A_3$ & $A_4$ & $A_5$ & $A_6$  & $A_7$   & $A_8$    & $A_9$  & $A_{10}$ & $A_{11}$ \\ \hline
Order of $W$ & $2$   & $6$   & $24$  & $120$ & $720$ & $5\,040$ & $40\,320$ & $362\,880$ & $10!$ & $11!$ & $12!$\\ \hline
Rational $u$  & $1$   & $3$   & $7$   & $25$  & $89$  & $379$  & $1\,679$  & $8\,289$   & $42\,561$ & $236\,099$ & $1\,357\,143$ \\ \hline
\end{tabular}
\caption{Count of rational Weyl group elements in type $A_{r}$ up to $r = 11$.}
\label{t:countA}
\end{table}

%INTERNAL COMMENT:
%It took my program around 24 hours to compute the number of rational Weyl group elements in A11. I think it's hopeless for me to determine A12, unless I learn more about the elements.

\begin{table}[ht]
\centering
\begin{tabular}{|l|l|l|l|l|l|l|l|l|}
\cline{1-7} \cline{9-9}
Type         & $D_5$    & $D_7$      & $D_9$          & $D_{11}$          & $D_{13}$          & $D_{15}$          & \,\, & $E_6$     \\ \cline{1-7} \cline{9-9} 
Order of $W$ & $1\,920$ & $322\,560$ & $9!\cdot 2^8$ & $11!\cdot 2^{10}$ & $13!\cdot 2^{12}$ & $15!\cdot 2^{14}$ & \,\, & $51\,840$ \\ \cline{1-7} \cline{9-9} 
Rational $u$ & $31$     & $127$      & $511$          & $2\,047$          & $8\,191$          & $32\,767$               & \,\, & $397$     \\ \cline{1-7} \cline{9-9} 
\end{tabular}
\caption{Count of rational Weyl group elements in types $D_r$, for $r \leq 15$ odd and $E_6$.}
\label{t:countDE}
\end{table}

\begin{remark}
    In the definition of $\Gamma(W)$, we join two rational Weyl group elements $u,v \in W$ iff there is a simple reflection $s$ such that $u = sv$. Alternatively, one could join two elements by an edge iff $u = vs$. However, the resulting graph is not connected unless it contains only $1$ vertex or the Lie type is $A_1$ or $A_2$.
\end{remark}

\begin{remark}
    It is interesting to see whether the notion of a rational Weyl group element can be extended to the non-crystallographic root systems
    of types $I_2(m)$ ($m\notin \{3,4,6\}$), $H_3$ and $H_4$. For type $I_2(m)$, we generate a root system from the generalized Cartan matrix of the form\footnote{This choice of the Cartan matrix for $I_2(m)$ includes the cases of $C_2$ and $G_2$ types; however, if one generates a root system of type $I_2(m)$ from a symmetric Cartan matrix, the results are not as nice (for the symmetric, for instance, $\Gamma(I_2(6))$ has $3$ vertices, which is misaligned with the result for $\Gamma(G_2)$, which has $1$ vertex; nevertheless, the results for $I_2(5)$ and $H_2$ agree whether the matrix is symmetric or not).} 
    \begin{equation*}
        \begin{bmatrix}
        2 & -4\cos^2(\pi/m)\\ -1 & 2
    \end{bmatrix}
    \end{equation*}
    For types $H_3$ and $H_4$, we choose a symmetric generalized Cartan matrix as in \cite[Section 2.13]{humphreys_r}. The only issue for testing the theory is the choice of a poset. If one defines a poset on the set of roots as
    \begin{equation*}
        \alpha \leq \beta \ \ \Leftrightarrow \ \ \beta-\alpha \in \mathbb{R}_{\geq 0}\Delta,
    \end{equation*}
    then the results are as follows. The rationality graph $\Gamma(I_2(m))$ is connected, and it contains $1$ vertex if $m$ is even and $3$ vertices if $m$ is odd; the graph $\Gamma(H_3)$ is not connected and contains $19$ vertices; the graph $\Gamma(H_4)$ is also not connected and contains $12\,440$ vertices. Another poset for $I_2(m)$ and $H_3$ was suggested by D. Armstrong in \cite[Section 5.4.1]{armstrong}. We verified that if one induces the Armstrong's poset in $H_3$ from $D_6$, then $\Gamma(H_3)$ contains $4$ vertices; however, the graph is not connected.
\end{remark}

% \footnote{Let us recall that $I_2(3) = A_2$, $I_2(4) = C_2$, $I_2(5) = H_2$ and $I_2(6) = G_2$. The Coxeter group of type $I_2(m)$ is the Dihedral group $D(2m)$, so it contains $2m$ elements. For the results in this remark, we view a root system $I_2(m)$ given by its generalized Cartan matrix as 
%     If one chooses simple roots for $I_2(m)$ of equal length, the results are not as symmetric. Next , we recall that the Coxeter group of type $H_3$ contains $120$ elements, and the Coxeter group of type $H_4$ contains $14\, 400$ elements; for the corresponding root systems, we use the description from \cite{humphreys_r}.} 

    %INTERNAL COMMENT: for H2, the Cartan matrix is chosen to be symmetric, the result is the same as for I2(5), but that's the only coincidence. If I2(6) is symmetric, then we get 3 vertices and a connected graph, which is not correct, because we want G2
 \section{Construction of solutions}\label{s:proofs1}
 In this section, we provide proofs of the results from Section~\ref{s:main_results1}.

\subsection{\texorpdfstring{Rationality graph $\Gamma(u)$ and $\nu$-sequence}{Rationality graph Gamma(u) and nu-sequence}}\label{s:nuseq}
In this subsection, for a given $u \in W$, we study properties of the $\nu$-sequence of $u$ (see Definition~\ref{d:nuseq}),  properties of its rationality graph $\Gamma(u)$ (see Definition~\ref{d:gammau}), and the relation between $\Gamma(u)$ and the $\nu$-sequence. We also prove the equivalence of conditions~\ref{i:rational1} and~\ref{i:rational2} of Definition~\ref{d:rational} (see Proposition~\ref{p:nu_nonrat}).

\begin{lemma}\label{l:nuk}
    For any $k \geq 0$, a root $\alpha \in \nu^0(u)$ belongs to $\nu^k(u)$ if and only if there is a path in $\Gamma(u)$ of length $k$ that starts at $\alpha$.
\end{lemma}
Note that we allow loops in $\Gamma(u)$, so a path can traverse the same vertex multiple times.
\begin{proof}
Let us prove the statement by induction. For $k=0$, the vertex set of $\Gamma(u)$ is the same as the set of paths of length $0$, and it is equal to $\nu^0(u)$. Assume that the statement holds for $k-1$, $k \geq 1$. By equation~\eqref{eq:nuseq}, a root $\alpha \in \nu^0(u)$ belongs to $\nu^k(u)$ if and only if there is a root $\beta \in \nu^{k-1}(u)$ such that $u^{-1}(\alpha) \leq \beta$; that is, there is an arrow $\alpha \rightarrow \beta$ in the graph $\Gamma(u)$. But then, there is a path of length $k-1$ that starts at $\beta$; thus, there is a path of length $k$ that starts at $\alpha$.
\end{proof}

\begin{proposition}\label{p:nu_descends}
    The sequence $\{\nu^k(u)\}_{k \geq 0}$ is descending and stabilizes after a finite number of terms.
\end{proposition}
\begin{proof}
By Lemma~\ref{l:nuk}, a root $\alpha \in \nu^0(u)$ belongs to $\nu^k(u)$ if and only if there is a path of length $k$ that starts at $\alpha$; hence, if $k \geq 1$ and $\alpha \in \nu^k(u)$, there is also a path of length $k-1$ that starts at $\alpha$, and thus $\alpha \in \nu^{k-1}(u)$. Since the elements of the sequence are subsets of a finite set, and since the sequence is descending, it stabilizes after a finite number of terms.
\end{proof}

The result of Proposition~\ref{p:nu_descends} allows us to define
\begin{equation}
    \nu(u) := \lim_{k \rightarrow \infty} \nu^k(u).
\end{equation}
As a consequence, the set $\nu(u)$ satisfies the following equation:
\begin{equation}\label{eq:nu}
    \nu(u) = u(\Adj \nu(u)) \cap \Pi_+.
\end{equation}

The following result shows that the two definitions of rationality of an element $u \in W$ are equivalent (see Definition~\ref{d:rational}).

\begin{proposition}\label{p:nu_nonrat}
A positive root $\alpha$ belongs to $\nu(u)$ if and only if there is a cycle in $\Gamma(u)$ that passes through $\alpha$. In particular, the graph $\Gamma(u)$ contains a cycle if and only if $\nu(u)$ is nonempty.
\end{proposition}
\begin{proof}
    If $\nu(u)$ is nonempty, we can find paths in $\Gamma(u)$ of arbitrary length, and since $\Gamma(u)$ has a finite number of vertices and is oriented, it contains a cycle. Conversely, if $\Gamma(u)$ contains a cycle, by Lemma~\ref{l:nuk}, for any $k \geq 1$, the set $\nu^k(u)$ is nonempty, and thus $\nu(u)$ is nonempty as well.
\end{proof}
\begin{proposition}\label{p:highrootnonrat}
    Let $\theta$ be the highest root. If $\theta \in \nu^0(u)$, then $u$ is not rational.
\end{proposition}
\begin{proof}
    Indeed, if $\theta \in \nu^0(u)$, then there is a loop in $\Gamma(u)$ at $\theta$, for $u^{-1}(\theta) \leq \theta$.
\end{proof}
\begin{proposition}\label{p:fixedpt}
    The graph $\Gamma(u)$ contains a loop if and only if there is a root $\alpha \in \Pi_+$ such that $\alpha \leq u(\alpha)$. In particular, if $u$ has a fixed point, then $\Gamma(u)$ contains a loop.
\end{proposition}
\begin{proof}
    Indeed, $\Gamma(u)$ has a loop if and only if there is $\beta \in \nu^0(u)$ such that $u^{-1}(\beta) \leq \beta$. For any $\beta \in \nu^0(u)$, there is $\alpha \in \Pi_+$ such that $\beta = u(\alpha)$; thus, $u^{-1}(\beta) \leq \beta$ if and only if $\alpha \leq u(\alpha)$.
\end{proof}

For the following corollary, identify the Weyl group of type $A_{n-1}$ with $S_{n}$. Although we are not going to use the next result, we find it curious enough to record here.

\begin{corollary}\label{c:cycles}
    For an element $u \in S_n$, the following statements hold:
    \begin{enumerate}[1)]
        \item If for some $k \in \{1,\ldots,n-1\}$, $u$ has the cycle $(1\,2\,\ldots\,k)$ in its cycle decomposition, then $u$ is not rational;\label{i:cycles1}
        \item If for some $k \in \{2,\ldots,n\}$, $u$ has the cycle $(k\,k+1\,\ldots\,n)$ in its cycle decomposition, then $u$ is not rational.\label{i:cycles2}
    \end{enumerate}
\end{corollary}
\begin{proof}
    \begin{enumerate}[1)]
        \item Assume $u$ has the cycle $(1\,2\,\ldots\, k)$ in its cycle decomposition, and let $j:=u(k+1)$. Then $u(\alpha_k) = \alpha_1 + \cdots + \alpha_{j-1} \geq \alpha_k$. By Proposition~\ref{p:fixedpt}, the graph $\Gamma(u)$ contains a loop, and thus $u$ is not rational.
        \item Similarly, if $u$ has the cycle $(k\,k+1\,\ldots\,n)$ in its cycle decomposition, then $u(\alpha_{k-1}) \geq \alpha_{k-1}$, and therefore, $u$ is not rational. \qedhere
    \end{enumerate}
\end{proof}
\begin{remark}
    If one views the special Coxeter element $C:=s_1\cdots s_{n-1}$ and $C^{-1}$ as elements of $S_n$, then they are given by the cycles $C = (1\, \ldots\, n)$ and $C^{-1} = (n\, \ldots\, 1)$. 
    These cycles correspond to the boundary cases of Corollary~\ref{c:cycles} ($k=n$ for Statement~\ref{i:cycles1} and $k = 1$ for Statement~\ref{i:cycles2}).
\end{remark}
\begin{remark}\label{r:howtonu0}
    The set $\nu^0(u)$ can be computed as follows. Write $u = vw_0$ for some $v \in W$, and decompose $v$ into its reduced decomposition $v = s_{i_1} s_{i_2} \cdots s_{i_m}$. Set
    \begin{equation}\label{eq:vbeta}
        \beta_1 := \alpha_{i_1}, \ \ \beta_k:= s_{i_1}s_{i_2}\cdots s_{i_{k-1}}(\alpha_k), \ \ k \geq 2.
    \end{equation}
    Then the set $\nu^0(u)$ is given by
    \begin{equation}\label{eq:nuvbeta}
        \nu^0(u) = \{\beta_1,\beta_2,\ldots,\beta_m\}.
    \end{equation}
    Indeed, $\nu^0(u)$ is the set of positive roots $\alpha$ such that $u^{-1}(\alpha)$ is positive; alternatively, it is the set of positive roots $\alpha$ such that $v^{-1}(\alpha)$ is negative, and the computation of the latter set via reduced decompositions is well-known. More generally, for any $w \in W$, the set $\nu^0(u\,|\,w)$ can be computed in the same way by considering a reduced decomposition of $uw^{-1}$.
\end{remark}

The last results of the subsection concern relative $\nu$-sequences (see Definition~\ref{d:nu_relative}).

\begin{proposition}\label{p:relroots}
    For any $u,v \in W$, set $w:= uv^{-1}$, and decompose $w$ into its reduced decomposition as $w = s_{i_1}\cdots s_{i_m}$. Then 
    \begin{equation}
        \{\alpha_{i_1},\ldots,\alpha_{i_m}\} \subseteq \Adj \nu^0(u\,|\,v) \subseteq \Pi_+ \cap \mathbb{Z}_+\{\alpha_{i_1},\ldots,\alpha_{i_m}\}.
    \end{equation}
\end{proposition}
\begin{proof}
    Write $\nu^0(u\,|\,v) = \{\beta_1,\ldots,\beta_m\}$ where $\beta_k$'s are given by~\eqref{eq:vbeta}. For each $j \in \{1,\ldots,m\}$, there is $k \leq j$ such that the indices $(i_1,\ldots,i_{k-1})$ do not contain $i_j$, and such that $i_{k} = i_{j}$; therefore, $\beta_k = s_{i_1}\cdots s_{i_{k-1}}(\alpha_{i_k}) \geq \alpha_{i_k}$.
\end{proof}

\begin{proposition}
    If $u \in W$ is rational, then for any $v \in W$, $\nu(u\,|\,v) = \emptyset$.
\end{proposition}
\begin{proof}
    Since $\nu^0(u\,|\,v) \subseteq \Pi_+$, we see that $\nu^1(u\,|\,v) \subseteq \nu^0(u)$. Reasoning inductively, we conclude that $\nu^{k+1}(u\,|\,v) \subseteq \nu^k(u)$; thus, if $u$ is rational, $\nu(u\,|\,v) = \emptyset$ for any $v \in W$.
\end{proof}
\begin{remark}
    For a given $u \in W$, the $\nu$-sequence of $u$ is descending (see Proposition~\ref{p:nu_descends}); however, relative $\nu$-sequences may not be descending. Consider, for example, type $A_3$ with $u:= s_{3}s_{1}s_2$ and $v:=s_{1}s_{2}$. One finds that $\nu^0(u\,|\,v) = \{\alpha_3\}$ and
    \begin{equation*}
        \nu^k(u\,|\,v) = \begin{cases}
            \{\alpha_1 + \alpha_2\} &\text{if} \ k \ \text{is odd};\\
            \{\alpha_2 + \alpha_3\}  &\text{if} \ k \ \text{is even.}
        \end{cases}
    \end{equation*}
\end{remark}

\begin{remark}
    Rationality of an element $u \in W$ can be linked to the nilpotency of the operator $T_{b,\overline{u}}$ described as follows. Let $\mathfrak{n}_\pm$ be the Lie algebra of $\mathcal{N}_\pm$,  $\pi_< : \mathfrak{g} \rightarrow \mathfrak{n}_-$ be the orthogonal projection. For an element $b \in \mathcal{B}_+$, a given $u \in W$ and its representative $\overline{u} \in N_G(\mathcal{H})$, define the operator $T_{b,\overline{u}} : \mathfrak{n}_+ \rightarrow \mathfrak{n}_+$ as
    \begin{equation*}
        T_{b,\overline{u}}(x) := \pi_{<}\Ad_{b\overline{u}}(x), \ \ x \in \mathfrak{n}_-.
    \end{equation*}
    It is not difficult to see that the image of $(T_{b,\overline{u}})^i$ satisfies the equation
    \begin{equation*}
        (T_{b,\overline{u}})^i(\mathfrak{n}_-) \subseteq \bigoplus_{\beta \in \Adj \nu^{i-1}(u)} \mathfrak{g}_{-\beta}, \ \ i \geq 1;
    \end{equation*}
    therefore, if $u$ is rational, $T_{b,\overline{u}}$ is nilpotent. We expect that the converse is also true: if $T_{b,\overline{u}}$ is nilpotent for a generic $b \in \mathcal{B}_+$, then $u$ is rational. However, we find no use of this map, except that it appears in the formula for the differential of the map $B$.
\end{remark}

\subsection{Proofs of Proposition~\ref{p:nbrat} and \ref{p:general} (General properties)}
We formulated Problem~\ref{problem} as the problem of finding a pair of rational maps $B : G \dashrightarrow \mathcal{B}_+$ and $N : G \dashrightarrow \mathcal{N}_-$. In this subsection, we show that, equivalently, we could formulate Problem~\ref{problem} as the problem of birationality of the map
    \begin{equation}\label{eq:nbrat_proof}
        \mathcal{T}_{\overline{u}}: \mathcal{N}_- \times \mathcal{B}_+ \rightarrow G, \ \ \mathcal{T}_{\overline{u}}(n,b) = nb\overline{u}n^{-1}, \ \ n \in \mathcal{N}_-, \ b \in \mathcal{B}_+,
    \end{equation}
for a fixed representative $\overline{u}$ of $u$. This equivalence follows from a simple observation in algebraic geometry (see Lemma~\ref{l:rat_section}). As a consequence, we derive Proposition~\ref{p:general}, which concerns some general properties of the maps $B$ and $N$. The statements of Proposition~\ref{p:general} are split into three statements in this subsection: Corollary~\ref{c:dominance}, Corollary~\ref{c:invariance} and Proposition~\ref{p:nfromb}.

\begin{lemma}\label{l:rat_section}
    Let $V$ and $W$ be irreducible algebraic varieties of the same dimension, and let $f : V \dashrightarrow W$ be a rational map. Assume there exists a rational map $s : W \dashrightarrow V$ such that $fs = \id$. Then $f$ is birational.
\end{lemma}
The lemma appeared before, for instance, in \cite[Lemma 6.2]{fg-moduli}. We provide a proof for convenience of the reader.
\begin{proof}
    Set $M:=\overline{s(W)}$ (Zariski closure) and $m:=\dim M$ and $n:= \dim V = \dim W$. From the equation $fs = \id$, we see that $f|_M : M \dashrightarrow W$ is a dominant map, and therefore, $m \geq n$; on the other hand, since $M \subseteq V$, $m \leq n$. We conclude that $m = n$, and since $M$ is closed, $M = V$. Hence both $s$ and $f$ are dominant maps, and therefore, both induce homomorphisms $f^*$ and $s^*$ on the fields of rational functions of $V$ and $W$; since $s^* \circ f^* = \id$, we conclude that $f^*$ is an isomorphism, and thus $f$ is a birational map.
\end{proof}

\begin{proof}[Proof of Proposition~\ref{p:nbrat}]
    Note that both the domain and the codomain of $\mathcal{T}_{\overline{u}}$ are irreducible quasi-affine varieties of the same dimension. The product of the maps $N$ and $B$ yields a rational section of $\mathcal{T}_{\overline{u}}$, and therefore, by Lemma~\ref{l:rat_section}, the map $\mathcal{T}_{\overline{u}}$ is birational.
\end{proof}

\begin{corollary}\label{c:dominance}
    For any solution $u \in W$ of Problem~\ref{problem}, the rational maps $B : G \dashrightarrow \mathcal{B}_+$ and $N : G \dashrightarrow \mathcal{N}_-$ are dominant.
\end{corollary}
\begin{proof}
    Indeed, since the product map $N \times B$ is inverse to $\mathcal{T}_{\overline{u}}$, its components $N$ and $B$ are dominant maps.
\end{proof}

\begin{corollary}\label{c:invariance}
    For any solution $u \in W$ of Problem~\ref{problem}, the corresponding rational maps $B$ and $N$ satisfy the following invariance properties: for any $n \in \mathcal{N}_-$ and a generic $g \in G$,
    \begin{align}
        &B(ngn^{-1}) = B(g);\label{eq:binvar_proof} \\
        &N(ngn^{-1}) = nN(g).
    \end{align}
\end{corollary}
\begin{proof}
    Given any $n \in \mathcal{N}_-$, define the rational maps $\bar{N}(g) := n^{-1}N(ngn^{-1})$ and $\bar{B}(g) := B(ngn^{-1})$. We see that
    \begin{equation*}
        ngn^{-1} = N(ngn^{-1}) B(ngn^{-1}) {\overline{u}} N(ngn^{-1})^{-1};
    \end{equation*}
    conjugating both sides by $n^{-1}(\cdot)n$, we see that the above equation can be rewritten as
    \begin{equation*}
        g = \bar{N}(g) \bar{B}(g) {\overline{u}} \bar{N}(g)^{-1}.
    \end{equation*}
    It follows that $\bar{N} \times \bar{B}$ is a rational section of the map $\mathcal{T}_{\overline{u}}$, and since $\mathcal{T}_{\overline{u}}$ is birational, $N(g) = \bar{N}(g)$ and $B(g) = \bar{B}(g)$.
\end{proof}

\begin{proposition}\label{p:nfromb}
For any solution $u$ of Problem \ref{problem} and its any representative ${\overline{u}}$, the rational map $N : G \dashrightarrow \mathcal{N}_-$ is given by
\begin{equation}\label{eq:nfromb_proof}
    N(g) = N_0(g) N_0(B(g){\overline{u}})^{-1} = [g\overline{w_0}^{-1}]_- ([B(g)\overline{u}\,\overline{w_0}^{-1}]_-)^{-1},
\end{equation}
where $\overline{w_0}$ is any representative of $w_0$.
\end{proposition}
\begin{proof}
    Indeed, decompose a generic $g \in G$ relative ${\overline{u}}$, and decompose $B(g){\overline{u}}$ relative $\overline{w_0}$ to obtain the equations
    \begin{align*}
        &g = N(g) B(g){\overline{u}} N(g)^{-1};\\
        &B(g){\overline{u}} = N_0(B(g){\overline{u}}) B_0(g)\overline{w_0} N_0(B(g){\overline{u}})^{-1}.
    \end{align*}
    Substituting $B(g){\overline{u}}$ from the second equation into the first equation, we see that~\eqref{eq:nfromb_proof} holds.
\end{proof}

\subsection{Preliminaries on Lie theory}\label{s:prelimlie}
For the next subsections, we set up notation and review some background on Lie theory. 

\paragraph{Gaussian decomposition.} For a generic $g \in G$, we denote by $g_- \in \mathcal{N}_-$ and $g_\oplus$ the unique elements such that $g = g_- g_{\oplus}$. For any $a \in \mathcal{H}$, $n_- \in \mathcal{N}_-$ and $n_+ \in \mathcal{N}_+$, the Gaussian decomposition has the following properties:
\begin{align*}
    &(n_-g)_- = n_-g_-,& &(gan_+)_- = g_-,& &(ag)_- = ag_- a^{-1};\\
    &(n_-g)_\oplus = g_\oplus,& &(gan_+)_\oplus = g_\oplus an_+,& &(ag)_{\oplus} = ag_\oplus.
\end{align*}
We use these properties routinely in Section~\ref{s:relative_existence}.

\paragraph{Cartan integers.} We use the following convention on Cartan integers $a_{ij}$:
\begin{equation*}%\label{eq:cartanints}
    a_{ij} := \langle \alpha_j,\alpha_i^\vee \rangle = 2\frac{\langle \alpha_j,\alpha_i\rangle}{\langle \alpha_i,\alpha_i\rangle}.
\end{equation*}
In this convention, the simple reflections act on simple roots as
\begin{equation}\label{eq:simple_refl}
    s_i(\alpha_j) = \alpha_j - a_{ij}\alpha_i.
\end{equation}

\paragraph{Chevalley generators.} Let $\{(e_i,h_i,f_i)\}_{i=1}^r$ be a collection of Chevalley generators of $\mathfrak{g}$, where $r$ is the rank of $\mathfrak{g}$ and $e_i := e_{\alpha_i}$, $\alpha_i \in \Delta$. We complement the Chevalley generators to a basis of $\mathfrak{g}$ by choosing arbitrary nonzero root vectors $e_{\beta} \in \mathfrak{g}_\beta$, $\beta \in \Pi$. Define the following $1$-parameter subgroups of $\mathfrak{g}$:
\begin{align*}
    &x_{\beta} := \exp(te_{\beta}),  \ \ t \in \mathbb{C};\\
    &x_{i} := x_{\alpha_i}, \ \ y_i := x_{-\alpha_i}, \ \ i \in \{1,\ldots,r\}.
\end{align*}
For each $i \in \{1,\ldots,r\}$, define a group homomorphism $\phi_i : \SL_2(\mathbb{C}) \rightarrow G$ via
\begin{equation*}
    \phi_i\begin{bmatrix}
        1 & t\\
        0 & 1 
    \end{bmatrix}:= x_i(t), \ \ 
    \phi_i\begin{bmatrix}
        1 & 0\\
        t & 1 
    \end{bmatrix}:= y_i(t).
\end{equation*}
Define the element $t^{h_i} \in \mathcal{H}$ via
\begin{equation*}
    t^{h_i} :=\phi_i\begin{bmatrix}
        t & 0\\
        0 & t^{-1} 
    \end{bmatrix}
\end{equation*}
The group versions $x_i$ and $y_i$ of  Chevalley generators satisfy various relations that are well-documented in \cite[Section 3]{schubertpos} and \cite[Section 2]{bruhatpos}; we will need some of them. For any $t_1,t_2 \in \mathbb{C}$ and $i \neq j$, 
\begin{equation*}
    x_i(t_1)y_j(t_2) = y_j(t_2) x_i(t_1);
\end{equation*}
\begin{equation}\label{eq:xiyicomm}
    x_i(t_1)y_i(t_2) = y_i\left( \frac{t_2}{1+t_1t_2}\right) (1+t_1t_2)^{h_i} x_i\left( \frac{t_1}{1+t_1t_2}\right) \ \ \ (t_1t_2 \neq -1);
\end{equation}
if $a_{ij} = 0$, then
\begin{align*}
    &x_i(t_1)x_j(t_2) = x_j(t_2) x_i(t_1);\\
    &y_i(t_1)y_j(t_2) = y_j(t_2) y_i(t_1);
\end{align*}
if $a_{ij} = a_{ji} = -1$ and $t_1 \neq -t_3$, then
    \begin{align*}
        &x_i(t_1)x_j(t_2)x_i(t_3) = x_j\left( \frac{t_2 t_3}{t_1 + t_3} \right)x_i(t_1+t_3) x_j \left( \frac{t_1 t_2}{t_1 + t_3} \right);\\
        &y_i(t_1)y_j(t_2)y_i(t_3) = y_j\left( \frac{t_2 t_3}{t_1 + t_3} \right)y_i(t_1+t_3) y_j \left( \frac{t_1 t_2}{t_1 + t_3} \right).
    \end{align*}
For any $\alpha_i\in \Delta$ and an element $a:= \exp(h) \in \mathcal{H}$, $h \in \mathfrak{h}$, we define the character $a \mapsto a^{\alpha_i}$ via
\begin{equation*}
    a^{\alpha_i} := e^{\alpha_i(h)};
\end{equation*}
then the generators $x_i$ and $y_i$ satisfy the following relations:
\begin{align*}
    &ax_i(t) = x_i(a^{\alpha_i}t)a;\\
    &ay_i(t) = y_i(a^{-\alpha_i}t)a.
\end{align*}

\paragraph{Special representatives.} For a simple reflection $s_i$, define $\dot{s}_i \in N_G(\mathcal{H})$ to be a representative given by
\begin{equation}\label{eq:dotsi}
    \dot{s}_i := x_i(-1) y_i(1) x_i(-1).
\end{equation}
It is well-known that the representatives $\dot{s}_i$ satisfy braid relations; therefore, $\dot{u} \in N_G(\mathcal{H})$ is well-defined. Our statements do not depend on the particular choice of $\overline{u} \in N_G(\mathcal{H})$, so we often choose~$\dot{u}$.

\subsection{Proof of Proposition~\ref{p:relative_existence} (Parabolic approximations)}\label{s:relative_existence}
Let us recall the setup of Proposition~\ref{p:relative_existence}. Fix any solution $v \in W$ of Problem~\ref{problem} and its representative $\overline{v}$, and denote by $\tilde{B} : G \dashrightarrow \mathcal{B}_+$ and $\tilde{N} : G \dashrightarrow \mathcal{N}_-$ the rational maps such that for a generic $g \in G$,
\begin{equation}\label{eq:decv_proof}
    g = \tilde{N}(g) \tilde{B}(g) \overline{v} \tilde{N}(g)^{-1}.
\end{equation}
For any other $u \in W$ and its representative $\overline{u}$, construct a sequence $\{P_k : G \dashrightarrow G\}_{k \geq 0}$ of rational maps via
\begin{equation}\label{eq:pk_proof}
\begin{aligned}
    &P_0(g):= \tilde{B}(g)\overline{v}\,\overline{u}^{-1};\\
    &P_k(g) := [P_{k-1}(g)]_{\oplus} \overline{u} [P_{k-1}(g)]_- \overline{u}^{-1}, \ \ k \geq 1.
\end{aligned}
\end{equation}
and construct a sequence of rational maps $\{N_k : G \dashrightarrow \mathcal{N}_-\}$ via
\begin{equation}
    N_k(g) := \tilde{N}(g) \tilde{N}(P_k(g)\overline{u})^{-1}, \ \ k \geq 0. 
\end{equation}
Below we show that if the sequence $\{P_k : G \dashrightarrow G\}_{k \geq 0}$ stabilizes after a finitely many terms, then $u$ is a solution of Problem~\ref{problem}. 

\begin{lemma}
    For any $k \geq 1$, the following equation holds:
    \begin{equation}\label{eq:npk}
        [P_{k-1}(g)]_-N_k(P_k(g)\overline{u}) = \tilde{N}(P_{k-1}(g)\overline{u}).
    \end{equation}
\end{lemma}
\begin{proof}
    Let us recall from Corollary~\ref{c:invariance} that for any $n \in \mathcal{N}_-$ and a generic $g \in G$, $\tilde{N}(ngn^{-1}) = n\tilde{N}(g)$. From equation~\eqref{eq:pk_proof}, we see that
    \begin{equation*}\begin{split}
        N_k(P_k(g)\overline{u}) = &N_k([P_{k-1}(g)]_\oplus \overline{u} [P_{k-1}(g)]_-) = \\ = &\tilde{N}\left( ([P_{k-1}(g)]_-)^{-1}[P_{k-1}(g)]_- [P_{k-1}(g)]_\oplus \overline{u} [P_{k-1}(g)]_-\right) =\\ = &([P_{k-1}(g)]_-)^{-1} \tilde{N}(P_{k-1}(g)\overline{u}).
        \end{split}
    \end{equation*}
    Thus equation~\eqref{eq:npk} holds.
\end{proof}

\begin{proposition}
    For any $k \geq 0$ and a generic $g \in G$, the following equation holds:
    \begin{equation}\label{eq:pkdec}
        g = N_k(g) P_k(g)\overline{u} N_k(g)^{-1}.
    \end{equation}
\end{proposition}
\begin{proof}
    Let us proceed by induction. For $k = 0$, the decomposition~\eqref{eq:pkdec} is precisely the decomposition~\eqref{eq:decv_proof} for $v$. Assuming that the decomposition exists for $k-1$, we see that
   \begin{equation*}\begin{split}
        g = N_{k-1}(g) &P_{k-1}(g)\overline{u} N_{k-1}(g)^{-1} =\\= &\tilde{N}(g) \tilde{N}(P_{k-1}(g)\overline{u})^{-1} P_{k-1}(g) \overline{u} \tilde{N}(P_{k-1}(g)\overline{u}) \tilde{N}(g)^{-1};
        \end{split}
    \end{equation*}
    using equation~\eqref{eq:npk}, we can rewrite the latter as
    \begin{equation*}\begin{split}
        g = \overbrace{\tilde{N}(g) \tilde{N}(P_{k}(g)\overline{u})^{-1}}^{N_k(g)} \underbrace{([P_{k-1}(g)]_-)^{-1} P_{k-1}(g)}_{[P_{k-1}(g)]_{\oplus}}\overline{u} [P_{k-1}(g)]_-\overbrace{\tilde{N}(P_k(g)\overline{u}) \tilde{N}(g)^{-1}}^{N_k(g)^{-1}} = \\ = N_k(g) P_{k}(g)\overline{u} N_k(g)^{-1}.
        \end{split}
    \end{equation*}
    Thus equation~\eqref{eq:pkdec} holds.
\end{proof}

\begin{corollary}\label{c:limit}
    Assume that there is $m \geq 0$ such that $P_m(g) \in \mathcal{B}_+$ for a generic $g$. Then the sequence $\{P_k : G \dashrightarrow G\}_{k \geq 0}$ stabilizes; that is, 
    \begin{equation*}
        P_m(g) = P_{m+1}(g) = \cdots
    \end{equation*}
    Moreover, setting $B(g):=P_m(g)$ and $N(g) := N_m(g)$, for a generic $g \in G$, the following decomposition holds:
    \begin{equation}\label{eq:pklim}
        g = N(g) B(g)\overline{u}N(g)^{-1}.
    \end{equation}
\end{corollary}
\begin{proof}
    Indeed, if $P_m(g) \in \mathcal{B}_+$ for a generic $g \in G$, then it follows from~\eqref{eq:pk_proof} that $P_{m+1}(g) = P_m(g)$, and so on for the next terms. Thus we see that~\eqref{eq:pklim} is the decomposition~\eqref{eq:pkdec} for $k := m$.
\end{proof}

% \begin{proposition}\label{p:relative_existence}
%     For any solution $v \in W$ of Problem~\ref{problem} and any $u \in W$ such that $\nu(u\,|\,v) = \emptyset$, the element $u$ is also a solution of Problem~\ref{problem}.
% \end{proposition}
We now turn to the proof of Proposition~\ref{p:relative_existence}, which states that if $\nu(u \, | \, v) = \emptyset$, then the sequences $\{P_k: G \dashrightarrow G\}_{k \geq 0}$ and $\{N_k:G \dashrightarrow \mathcal{N}_-\}_{k \geq 0}$ stabilize, and hence $u$ is a solution of Problem~\ref{problem}.

\begin{proof}[Proof of Proposition~\ref{p:relative_existence}]
    We employ the notation that was set up in Section~\ref{s:prelimlie}. The statement does not depend on the given representatives of $u$ and $v$, so we choose the ones determined by~\eqref{eq:dotsi}.
    By Corollary~\ref{c:limit}, it suffices to show that there is $m \geq 0$ such that for a generic $g \in G$, $P_m(g) \in \mathcal{B}_+$. Let us set
    \begin{equation}
        \Delta^k := \Delta \cap \Adj \nu^k(u\,|\,v).
    \end{equation}
    Decomposing $uv^{-1}$ into its reduced decomposition $uv^{-1} = s_{i_1}\cdots s_{i_m}$, we see that, by Proposition~\ref{p:relroots}, 
    \begin{equation}
        \Delta^0 = \{\alpha_{i_1},\ldots,\alpha_{i_m}\} \subseteq \Adj \nu^0(u\,|\,v).
    \end{equation}
    Let $\mathcal{P}_k$ be the parabolic subgroup that is generated by $x_{\beta}(t)$, $\beta \in \Pi_+ \cup (\Pi_- \cap \mathbb{Z}\Delta^k)$. We see that $P_0(g)$ is the product of elements from $\mathcal{B}_+$ and $y_{i_j}(1)$, $j \in \{1,\ldots,m\}$ (here we use formula~\eqref{eq:dotsi} for the specific choices of $\overline{u}$ and $\overline{v}$); hence, $P_0(g)\overline{u} \in \mathcal{P}_0$ for a generic $g \in G$. Assuming that $P_{k-1}(g) \in \mathcal{P}_{k-1}$, we see that $[P_{k-1}(g)]_-$ can be written as the product of some $x_{-\alpha}(t_{\alpha})$ for $\alpha \in \Delta^{k-1}$; therefore, $\overline{u}[P_{k-1}(g)]_-\overline{u}^{-1}$ is the product of $x_{-u(\alpha)}(c_{\alpha}t_{\alpha})$ for $\alpha \in \Delta^{k-1}$ and some $c_{\alpha} \in \mathbb{C}$, and if $u(\alpha) > 0$, we see that $u(\alpha) \in \nu^{k}(u\,|\,v)$. The latter implies that if $u(\alpha)>0$, then $u(\alpha) \in \mathbb{Z}_+ \Delta^k$, and in particular, $x_{-u(\alpha)}(t)$ can be written as the product of some $x_{-\beta}(t_\beta)$ for $\beta \in \Delta^k$. We conclude that $P_k(g) \in \mathcal{P}_k$, and since $\nu(u\,|\,v) = \emptyset$, by Corollary~\ref{c:limit}, there is $m \geq 0$ such that $P_m(g) \in \mathcal{B}_+$. Thus $u$ is a solution of Problem~\ref{problem}.
\end{proof}

\begin{corollary}\label{c:rat_sol}
    Any rational Weyl group element is a solution of Problem~\ref{problem}.
\end{corollary}
\begin{proof}
    Indeed, for rational Weyl group elements $u \in W$, $\nu(u) = \nu(u\,|\,w_0) = \emptyset$.
\end{proof}

\begin{corollary}\label{c:contrapositive}
    Assume that $u$ is not a solution of Problem~\ref{problem}. Then for any $v \in W$ such that $\nu(u\,|\,v) = \emptyset$, $v$ is not a solution of Problem~\ref{problem} as well. 
\end{corollary}

The next corollary corresponds to Statement~\ref{i:simple_nonexist1} of Proposition~\ref{p:simple_nonexist}.

\begin{corollary}\label{c:negative_notsolution}
    Let $u \in W$ be an element such that $u(\alpha) \in \Pi_-$ and $u$ is not a solution of Problem~\ref{problem}. Then $s_{\alpha}u$ is not a solution of Problem~\ref{problem} as well.
\end{corollary}

\begin{proof}
    Computing the relative $\nu$-sequence, we see that
    \begin{align*}
        &\nu^0(u\, | \,s_{\alpha}u) = \{\alpha\};\\
        &\nu^1(u\, | \,s_{\alpha}u) = \{u(\alpha)\} \cap \Pi_+ = \emptyset.
    \end{align*}
By Corollary~\ref{c:contrapositive}, $s_{\alpha}u$ is not a solution.
\end{proof}

\begin{remark}
    For any $u \in W$, one can also show the invariance of each term of the sequences $\{N_k : G \dashrightarrow \mathcal{N}_-\}_{k \geq 0}$ and $\{P_k : G \dashrightarrow G\}_{k \geq 0}$:
    \begin{align*}
        &P_k(ngn^{-1}) = P_k(g);\\
        &N_k(ngn^{-1}) = nN_k(g).
    \end{align*}
\end{remark}

\subsection{Proof of Proposition~\ref{p:simple_nonexist} (Nonexistence of solutions)}\label{s:nonexistence}
In this subsection, we develop methods for showing that an element $u \in W$ is not a solution of Problem~\ref{problem}. One such method can be derived from the contrapositive of Proposition~\ref{p:relative_existence} (see Corollary~\ref{c:contrapositive} and Corollary \ref{c:negative_notsolution}); the other method consists of analyzing the fibers of the pair of maps $(B_0,N_0)$ for the solution $w_0$. In particular, we prove statements \ref{i:simple_nonexist2} and \ref{i:simple_nonexist3} of Proposition~\ref{p:simple_nonexist}, which are split into three other statements: Proposition~\ref{p:fixer_notsolution}, Proposition~\ref{p:orthosimpcycle} and Proposition~\ref{p:sAnonex}. Throughout the subsection, we employ the notation that was set up in Section~\ref{s:prelimlie}.

\begin{lemma}\label{l:fibers}
    Let $u \in W$ be any Weyl group element and $\overline{u}$ be its representative. Assume that for a generic $b \in \mathcal{B}_+$, there exists $n:= n(b) \in \mathcal{N}_-$ such that $n \neq 1$ and $n^{-1}b\overline{u}n\overline{u}^{-1} \in \mathcal{B}_+$. Then $u$ is not a solution of Problem~\ref{problem}.
\end{lemma}
\begin{proof}
    Indeed, pick any representative $\overline{w_0}$ of $w_0$ and the pair of maps $(B_0,N_0)$ as in \eqref{eq:w0n0}-\eqref{eq:w0b0}. Consider the rational map $L_0 : \mathcal{B}_+ \overline{u} \dashrightarrow \mathcal{B}_+ \overline{w_0}$ given by $L_0(b\overline{u}) = B_0(b\overline{u})\overline{w_0}$. 
    Observe that if $L_0(b\overline{u}) = L_0(b^\prime \overline{u})$, then 
    \begin{align*}
        &b\overline{u} = N_0(b\overline{u}) L_0(b\overline{u}) N_0(b\overline{u})^{-1},\\
        &b^\prime \overline{u} = N_0(b^\prime \overline{u}) L_0(b^\prime \overline{u}) N_0(b^\prime \overline{u})^{-1},
    \end{align*}
    and therefore, by setting $n:=N_0(b\overline{u}) N_0(b^\prime \overline{u})^{-1}$, we arrive at the relation
    \begin{equation*}
        b^\prime = n^{-1} b\overline{u}n\overline{u}^{-1} \in \mathcal{B}_+.
    \end{equation*}
    Clearly, $b^\prime \neq b$ if and only if $n \neq 1$.
    
    Now, let us show that if $u$ is a solution of Problem~\ref{problem}, then the generic fiber of $L_0$ is necessarily of size $1$. Let $N : G \dashrightarrow \mathcal{N}_-$ and $B : G \dashrightarrow \mathcal{B}_+$ be the pair of rational maps that give rise to the rational decomposition with $\overline{u}$. Define the map $L : \mathcal{B}_+\overline{w_0} \dashrightarrow \mathcal{B}_+\overline{u}$ via $L(b\overline{w_0}) = B(b\overline{w_0})\overline{u}$. Then the rational maps $L$ and $L_0$ are inverse to each other. Indeed,
    \begin{equation*}
        L_0(L(b\overline{w_0})) = B_0(B(b\overline{w_0})\overline{u})\overline{w_0} = B_0(N(b\overline{w_0})B(b\overline{w_0})\overline{u}N(b\overline{w_0})^{-1})\overline{w_0} = b\overline{w_0}
    \end{equation*}
    where we used the invariance property~\eqref{eq:binvar}, and likewise for $L(L_0(b\overline{u})) = b\overline{u}$.
\end{proof}

\begin{proposition}\label{p:fixer_notsolution}
    Let $u \in W$ be an element such that $u(\alpha) = \alpha$ for some simple root $\alpha$. Then $u$ is not a solution of Problem~\ref{problem}.
\end{proposition}
\begin{proof}
    Set $x(t):=x_{\alpha}(t)$ and $y(t):=x_{-\alpha}(t)$, $t \in \mathbb{C}$. Decompose a generic element $b \in \mathcal{B}_+$ as $b = ax_{\alpha}(t)n^\prime$ where $n^\prime \in \mathcal{N}_+$ is the product of $x_{\beta}(t_{\beta})$ for $\beta \in \Pi_+ \setminus \{\alpha\}$. Let $r \in \mathbb{C}$ be a number such that for a given representative $\overline{u}$ of $u$ and any $t \in \mathbb{C}$, $\overline{u}y(t)\overline{u}^{-1} = y(rt)$. Set
    \begin{equation}\label{eq:nfixer}
        n:=y(c), \ \ c := \frac{ra^{-\alpha}-1}{rt}.
    \end{equation}
    We claim that $nb\overline{u}n\overline{u}^{-1} \in \mathcal{B}_+$. Indeed, the inclusion equation is given by
    \begin{equation*}
        y(-c) ax(t) n^\prime y(rc) \in \mathcal{B}_+;
    \end{equation*}
    commuting $n^\prime$ with $y(rc)$ (here we use the fact that $\alpha$ does not appear in $n^\prime$) and $y(-c)$ with $a$, the equation is equivalent to 
    \begin{equation*}
        y(-a^{-\alpha}c) ax(t) y(rc).
    \end{equation*}
    Applying the relation~\eqref{eq:xiyicomm}, we can further rewrite the equation as
    \begin{equation*}
        y(-a^{-\alpha}c) y\left( \frac{rc}{1+rtc}\right) \in \mathcal{B}_+
    \end{equation*}
    or
    \begin{equation*}
        -a^{\alpha}c + \frac{rc}{1+rtc} = 0.
    \end{equation*}
    Now it is a matter of direct check that $c$ given by equation~\eqref{eq:nfixer} is a solution of the above equation. By Lemma~\ref{l:fibers}, $u$ is not a solution of Problem~\ref{problem}.
\end{proof}

\begin{proposition}\label{p:orthosimpcycle}
    Assume there is a cycle in $\Gamma(u)$ that consists of a collection of pairwise orthogonal simple roots. Then $u$ is not a solution of Problem~\ref{problem} and is not rational.
\end{proposition}
\begin{proof}
    Let $\Delta^\prime:=\{\alpha_{i_j} \}_{j \in I} \subseteq \Delta\cap \nu^0(u)$ be a collection of simple roots enumerated by the cyclic group $I:=\mathbb{Z}_m$, such that $a_{i_j,i_{j+1}} = 0$. We assume that $m > 1$ (the case of $m=1$ is covered by Proposition~\ref{p:fixer_notsolution}), and that there is a cycle
    \begin{equation*}
        u^{-1}(\alpha_{i_{j-1}}) \leq \alpha_{i_j}, \ \ j \in I.
    \end{equation*}
    Clearly, the above equation implies that $u(\alpha_{i_j}) = \alpha_{i_{j-1}}$. For a given representative $\overline{u}$ of $u$, let $\{g_{j} \in \mathbb{C}\}_{j\in I}$ be a collection of numbers such that
    \begin{equation*}
        \overline{u}y_{i_{j}}(t)\overline{u}^{-1} = y_{i_{j-1}}(g_{j}t), \ \ t \in \mathbb{C}.
    \end{equation*}
    Decompose a generic element $b \in \mathcal{B}_+$ as
    \begin{equation*}
        b = a \left[\prod_{j\in I} x_{i_j}(t_{j})\right]n^\prime
    \end{equation*}
    where $n^\prime$ is the product of $x_{\beta}(t_{\beta})$ for $\beta \in \Pi_+ \setminus \Delta^\prime$, in any order. Writing $n \in \mathcal{N}_-$ as
    \begin{equation*}
        n = \prod_{j \in I} y_{i_j}(c_{j}),
    \end{equation*}
    the objective is to use Lemma \ref{l:fibers} and to show that the inclusion equation $n^{-1}b\overline{u}n\overline{u}^{-1} \in \mathcal{B}_+$ has a nontrivial solution in terms of $c_j$, $j \in I$. Indeed, we see that
    \begin{equation*}\begin{split}
        n^{-1}b\overline{u}n\overline{u}^{-1} = &\left[\prod_{j \in I} y_{i_j}(-c_{j})\right]a \left[\prod_{j\in I} x_{i_j}(t_{j})\right]n^\prime\prod_{j \in I} y_{i_{j-1}}(g_{j}c_{j}) = \\ = &a\left[\prod_{j \in I} y_{i_j}(-a^{\alpha_{i_j}}c_{j})\right]\left[\prod_{j\in I} x_{i_j}(t_{j})\right]\left[\prod_{j \in I} y_{i_{j-1}}(g_{j}c_{j})\right]n^\pprime,
        \end{split}
    \end{equation*}
    where $n^{\pprime} \in \mathcal{N}_+$ is some element. Since the simple roots in $\Delta^\prime$ are pairwise orthogonal, the inclusion equation is equivalent to the equation
    \begin{equation*}
        \prod_{j \in I} y_{i_j}(-a^{\alpha_{i_j}}c_{j})x_{i_j}(t_{j})y_{i_{j}}(g_{j+1}c_{j+1}) \in \mathcal{B}_+.
    \end{equation*}
    Commuting $x_{i_j}$ with $y_{i_j}$ via~\eqref{eq:xiyicomm}, the above equation is equivalent to
    \begin{equation*}
        \prod_{j \in I} y_{i_j}(-a^{\alpha_{i_j}}c_{j})y_{i_j}\left(\frac{g_{j+1}c_{j+1}}{1+t_jg_{j+1}c_{j+1}} \right) \in \mathcal{B}_+.
    \end{equation*}
    Setting $a_j := a^{\alpha_{i_j}}$, $j \in I$, the above equation is equivalent to the system of equations
    \begin{equation*}
        -a_jc_{j} + \frac{g_{j+1}c_{j+1}}{1+t_jg_{j+1}c_{j+1}} = 0, \ \ j \in I.
    \end{equation*}
    The nontrivial solution to the above system is given as follows. Set
    \begin{equation*}
        D_j := t_j \prod_{k \neq j}g_k + t_{j+1} a_{j+1} \prod_{k \notin \{j,j+1\}} g_k + t_{j+2} a_{j+1}a_{j+2}\prod_{k \notin \{j,j+1,j+2\}} g_k + \cdots + t_{j-1} \prod_{k \neq j}a_k
    \end{equation*}
    where we take advantage of the fact that $I = \mathbb{Z}_m$, a cyclic group; then the nontrivial solution for a generic $b$ is given by
    \begin{equation*}
       c_j = \frac{\prod_{k \in I}g_k - \prod_{k \in I}a_k}{a_jg_jD_j}, \ \ j \in I.
    \end{equation*}
    By Lemma~\ref{l:fibers}, we conclude that $u$ is not a solution of Problem~\ref{problem}.
\end{proof}

\begin{proposition}\label{p:sAnonex}
    Let $u \in W$ be an element with the following properties: 1) there are simple roots $\alpha_i$ and $\alpha_j$ with $a_{ij}=a_{ji} = -1$ such that $\alpha_i,\alpha_i+\alpha_j \in \nu^0(u)$; 2) there is a $2$-cycle on the vertices $\alpha_i$ and $\alpha_i + \alpha_j$ in $\Gamma(u)$. Then $u$ is not a solution of Problem~\ref{problem}.
\end{proposition}
\begin{proof}
    Under the above conditions, the existence of the $2$-cycle implies that $u$ acts as a simple reflection $s_j$ on the $A_2$ root system spanned by $\alpha_i$ and $\alpha_j$; that is, 
    \begin{align*}
        &u(\alpha_i) = u^{-1}(\alpha_i) = \alpha_i + \alpha_j;\\
        &u(\alpha_j) = -\alpha_j.
    \end{align*} 
    In other words, $s_ju(\alpha_i) = \alpha_i$ and $s_ju(\alpha_j) = \alpha_j$. Choose $\overline{s_i}:=\dot{s}_i$, as in~\eqref{eq:dotsi}. Set
    \begin{equation*}
        n:= y_i(c_1)y_j(c_2)y_i(c_3), \ \ c_1,c_2,c_3 \in \mathbb{C},
    \end{equation*}
    and let $r_i,r_j \in \mathbb{C}$ be the numbers such that for any $t \in \mathbb{C}$,
    \begin{align*}
        &\overline{s_ju}y_i(t)\overline{s_ju}^{-1} = y_i(r_it),\\
        &\overline{s_ju}y_j(t)\overline{s_ju}^{-1} = y_j(r_jt).
    \end{align*}
    Therefore,
    \begin{align*}
        \overline{s_ju} n \overline{s_ju}^{-1} = y_i(r_ic_1)y_j(r_jc_2)y_i(r_ic_3);
    \end{align*}
    conjugating both sides by $\overline{s_j}^{-1}(\cdot)\overline{s_j}$, we see that
    \begin{equation*}
        \overline{u} n \overline{u}^{-1} = \overline{s_j}^{-1}y_i(r_ic_1)y_j(r_jc_2)y_i(r_ic_3)\overline{s_j}.
    \end{equation*}
    The right-hand side can be easily computed in $\SL_3(\mathbb{C})$, and in the end, we arrive at
    \begin{equation*}
        \overline{u} n \overline{u}^{-1} = y_j\left(-\frac{c_1 + c_3}{r_jc_2c_3} \right)y_i(r_ir_jc_2c_3) y_j\left(\frac{c_1 + c_3}{r_jc_2c_3}\right) x_j(-r_j c_2).
    \end{equation*}
    Set 
    \begin{align*}
        &w_1:=w_1(c_1,c_2,c_3):=-\frac{c_1+c_3}{r_jc_2c_3};\\
        &w_2:=w_2(c_1,c_2,c_3):=r_ir_jc_2c_3.
    \end{align*}
    Decompose a generic element $b \in \mathcal{B}_+$ as
    \begin{equation*}
        b = ax_i(t_1)x_j(t_2)x_i(t_1)n^\prime
    \end{equation*}
    where $n^\prime$ is the product of $x_{\beta}(t_{\beta})$, $\beta \in \Pi_+ \setminus\{\alpha_i,\alpha_j,\alpha_i+\alpha_j\}$ taken in any order. Performing algebraic manipulations similar to the ones in the proof of Proposition~\ref{p:orthosimpcycle}, the inclusion equation $n^{-1}b\overline{u}n\overline{u}^{-1}\in \mathcal{B}_+$ is equivalent to the equation
    \begin{equation*}
        y_i(-a^{\alpha_i}c_3)y_j(-a^{\alpha_j}c_2)y_i(-a^{\alpha_i}c_1)x_i(t_1)x_j(t_2)x_i(t_1)y_i(w_1)y_j(w_2)y_i(-w_1) \in \mathcal{B}_+.
    \end{equation*}
    This equation can be solved for $(c_1,c_2,c_3)$ in $\SL_3(\mathbb{C})$. Let us state one of the nontrivial solutions; set
    \begin{equation*}
        \Delta:=a^{2\alpha_j}t_2^2 - 4r_ja^{\alpha_j}.
    \end{equation*}
    Then 
    \begin{align*}
        &c_1 = \frac{1}{a^{\alpha_i}(t_{1} + t_{3})};\\
        &c_2 = \frac{(t_{1} + t_{3})(2r_{i}r_{j} + a^{\alpha_i}\Delta^{1/2} + a^{\alpha_i}a^{\alpha_j}t_{2})}{2a^{\alpha_j}r_{j}(a^{\alpha_i}(t_{1} + t_{3}) + r_{i}t_{2}t_{3})};\\
        &c_3 =  \frac{-2a^{\alpha_i}r_{i}(t_{1} + t_{3}+ a^{\alpha_j}t_{1}t_{2}^2t_{3} - r_{j}(t_1+t_3)^2)}{a^{\alpha_i}(t_{1} + t_{3}) + t_{3}\Delta^{1/2}(r_{i}t_{2} t_{1} +1) + a^{\alpha_j}t_{2}(t_1-t_3)}.
    \end{align*}
    Thus the statement holds.
\end{proof}

% \begin{remark}
%     In context of Proposition~\ref{p:fixer_notsolution}, a formula for $n$ in terms of matrix entries of $b$ can be quite complicated even in the simplest cases. For instance, formulas for simple reflections in type $A_2$ involve radicals (see equations~\eqref{eq:s1a2_1}-\eqref{eq:s1a2_4}), and in $C_2$ or $G_2$, the resulting equations cannot be solved in radicals at all. 
% \end{remark}

% \begin{proposition}\label{p:negative_notsolution}
%     Let $u \in W$ be any element such that $u(\alpha) \in \Pi_- \setminus \{-\alpha\}$. Then $u$ is not a solution (of Problem~\ref{problem}) if and only if $s_{\alpha}u$ is not a solution.
% \end{proposition}

% \begin{proof}
%     Computing relative $\nu$-sequences, we see that
%     \begin{align}
%         &\nu^0(u\, | \,s_{\alpha}u) = \{\alpha\} = \nu^0(s_{\alpha}u\, | \, u);\\
%         &\nu^1(u\, | \,s_{\alpha}u) = \{u(\alpha)\} \cap \Pi_+ = \emptyset;\\
%         &\nu^1(s_{\alpha}u\, | \,u) = \{s_{\alpha}u(\alpha)\} \cap \Pi_+ = \emptyset.
%     \end{align}
% By Corollary~\ref{c:contrapositive}, $u$ is not a solution if and only if $s_{\alpha}u$ is not a solution.
% \end{proof}

\subsection{Proof of Theorem~\ref{thm1}}

In this subsection, we provide a proof of Theorem~\ref{thm1}. One direction of the theorem is contained in Corollary~\ref{c:rat_sol}: if $u \in W$ is rational, then $u$ is a solution of Problem~\ref{problem}. For the second part of the theorem, the case of $\ell(u) = \ell(w_0)-1$ is treated in Proposition~\ref{p:len1case}, and the case of $\ell(u) = \ell(w_0) - 2$ is treated in Proposition~\ref{p:len2case}. 

\begin{proposition}\label{p:len1case}
    An element $s_iw_0$ is rational if and only if $s_iw_0$ is a solution of Problem~\ref{problem}, if and only if $w_0(\alpha_i) \neq -\alpha_i$;
\end{proposition}
\begin{proof}
    We see that the rationality graph $\Gamma(s_iw_0)$ consists of a single vertex $\alpha_i$; hence, $s_iw_0$ is not rational if and only if there is a loop at $\alpha_i$, if and only if $(s_iw_0)^{-1}(\alpha_i) = \alpha_i$, if and only if $w_0(\alpha_i) = -\alpha_i$. By Proposition~\ref{p:fixer_notsolution}, if $s_iw_0(\alpha_i) = \alpha_i$, then $s_iw_0$ is not a solution; by Corollary~\ref{c:rat_sol}, if $s_iw_0$ is not a solution, it is not rational, and by the above, $w_0(\alpha_i) = -\alpha_i$.
\end{proof}

\begin{proposition}\label{p:len2case}
    For $i \neq j$, an element $s_i s_j w_0$ is rational if and only if $s_i s_j w_0$ is a solution of Problem~\ref{problem}, if and only if the following  conditions are satisfied: 1) $w_0(\alpha_j) \notin \{-\alpha_i,-\alpha_j\}$; 2) If $a_{ij} = 0$, then $w_0(-\alpha_i) \neq \alpha_i$.
\end{proposition}
\begin{proof}
         We see that the vertex set of $\Gamma(s_is_jw_0)$ is given by
        \begin{equation*}
            \nu^0(s_is_jw_0) = \{\alpha_i, s_i(\alpha_j)\}.
        \end{equation*}
        If $s_is_jw_0$ is not rational, then there is either a loop at $\alpha_i$, a loop at $s_i(\alpha_j)$, or a $2$-cycle between the two vertices. A simple computation shows that the loops and the $2$-cycle appear under the following circumstances:
        \begin{description}
            \item[Loop at $\alpha_i$.] That is,
            \begin{equation*}
                (s_is_jw_0)^{-1}(\alpha_i) \leq \alpha_i,
            \end{equation*}
            which is equivalent to the statement $w_0(\alpha_i) = -\alpha_i$ and $a_{ij} = 0$. Under these conditions, $s_is_jw_0(\alpha_i) = \alpha_i$; hence, by Proposition~\ref{p:fixer_notsolution}, $s_is_jw_0$ is not a solution.
            \item[Loop at $s_i(\alpha_j)$.] That is, the following equation holds:
            \begin{equation}\label{eq:sisjw0}
                (s_is_jw_0)^{-1}(s_i(\alpha_j)) \leq s_i(\alpha_j).
            \end{equation}
            Then either a) $w_0(\alpha_j)=-\alpha_j$, or b) $w_0(\alpha_j) = - \alpha_i$ and $a_{ij} \neq 0$. Indeed, using equation~\eqref{eq:simple_refl}, we can expand \eqref{eq:sisjw0} as
            \begin{equation*}
                w_0(-\alpha_j) \leq \alpha_j - a_{ij} \alpha_i,
            \end{equation*}
            and then it's easy to see the equivalence. In case of a), $s_jw_0$ is not a solution by Proposition~\ref{p:len1case}, and since $s_jw_0(\alpha_i) \in \Pi_-$, by Corollary~\ref{c:negative_notsolution}, $s_is_jw_0$ is not a solution. If b) is the case, then $\Gamma(s_is_jw_0)$ also has a $2$-cycle (see the next item).
            \item[A $2$-cycle.] This is the case if and only if $w_0(\alpha_j) = -\alpha_i$. Indeed, there is a $2$-cycle if and only if the following system of inequalities is satisfied:
            \begin{align*}
                    &(s_is_jw_0)^{-1}(\alpha_i) \leq s_i(\alpha_j);\\
                    &(s_is_jw_0)^{-1}(s_i(\alpha_j)) \leq \alpha_i.
            \end{align*}
            We see that the second equation is satisfied if and only if $w_0(\alpha_j) = -\alpha_i$, and in this case, the Lie type of $W$ is simply laced, so $a_{ij} = -a_{ji}$; hence, $-w_0(s_j(\alpha_i)) = s_i(\alpha_j)$, and the first equation is satisfied. Now, if $a_{ij} = 0$, then there is a cycle between two orthogonal simple roots, and by Proposition~\ref{p:orthosimpcycle}, $s_is_jw_0$ is not a solution. If $a_{ij} = -1$, then $s_is_jw_0$ is not a solution by Proposition~\ref{p:sAnonex}.
        \end{description}
        Thus the statement holds.
\end{proof}
\begin{corollary}
    Assume that the Lie type is such that $w_0 = -\id$, and let $u \in W$ be any element of length $\ell(u) \geq \ell(w_0) - 2$. Then $u$ is a solution of Problem~\ref{problem} if and only if $u$ is rational, if and only if $u = w_0$.
\end{corollary}
\begin{proof}
    Indeed, these are the Lie types precisely in which $w_0(\alpha_i) = -\alpha_i$ for every simple root $\alpha_i$. The corollary follows from Proposition~\ref{p:len2case}.
\end{proof}
\begin{corollary}\label{c:typeA2}
    In type $A_2$, an element $u \in W$ is a solution of Problem~\ref{problem} if and only if $u$ is rational. In particular, the set of solutions of Problem~\ref{problem} in type $A_2$ is given by 
    \begin{equation*}%\label{eq:a2_solution}
        \{w_0 = s_1s_2s_1,C:=s_1s_2,C^{-1} = s_2 s_1\}.
    \end{equation*}
\end{corollary}

 \section{Combinatorial results}\label{s:proofs2}
In this section, we provide proofs of the results from Section~\ref{s:main_results2}. Throughout the section, we assume the root system of $G$ is indecomposable.

\subsection{Proof of Theorem~\ref{thm2}}
In this subsection, we prove Theorem~\ref{thm2}. We split the statement of the theorem into three separate parts: Proposition~\ref{p:left_graph}, Corollary~\ref{c:morethanonevert} and Proposition~\ref{p:zmod2symmetry}.

\begin{lemma}
    For any $u \in W$ and any simple root $\alpha$,
    \begin{equation}\label{eq:nu0salpha}
        \nu^0(s_{\alpha}u) = s_{\alpha}(\nu^0(u)\setminus\{\alpha\}) \cup (s_{\alpha}u(\Pi_+) \cap \{\alpha\}).
    \end{equation}
\end{lemma}
\begin{proof}
    Indeed,
    \begin{equation*}\begin{split}
        \nu^0(s_{\alpha}u) = s_{\alpha}u(\Pi_+) \cap \Pi_+ = &s_{\alpha}\left( u(\Pi_+) \cap (\Pi_+\setminus \{\alpha\}\cup \{-\alpha\})\right) = \\ = &s_{\alpha}(\nu^0(u)\setminus\{\alpha\}) \cup (s_{\alpha}u(\Pi_+) \cap \{\alpha\}).
        \end{split}
    \end{equation*}
    Thus the equation holds.
\end{proof}

For the next lemma, given a simple root $\alpha \in \Delta$ and a root $\beta \in \Pi$, denote by $c_{\alpha}(\beta)$ the coefficient of $\beta$ at $\alpha$ in its expansion of simple roots; that is, $\beta = \sum_{\alpha \in \Delta} c_{\alpha}(\beta)\alpha$.

\begin{lemma}\label{l:correct_simple}
    Let $u \in W$ be a rational Weyl group element such that $u \neq w_0$. Then there exists a simple root $\alpha$ such that $u^{-1}(\alpha) > 0$ and $u(\alpha) < 0$.
\end{lemma}
\begin{proof}
    Since $u \neq w_0$ and $u$ is rational, there is $m \geq 1$ such that $\nu^m(u) \neq \emptyset$ and $\nu^{m+1}(u) = \emptyset$. Pick any $\beta \in \nu^m(u)$; since $u^{-1}(\beta) > 0$, there is a simple root $\alpha$ such that $\alpha \leq \beta$ and $u^{-1}(\alpha) > 0$. Hence $\alpha \in \Adj \nu^m(u)$, and since $\nu^{m+1}(u) = \emptyset$, we see that $u(\alpha) < 0$. Thus $\alpha$ is the required root.
\end{proof}

% \begin{lemma}\label{l:betalambda}
%     Let $\beta \in \Pi_+$, $\alpha \in \Delta$, $\lambda \in \mathbb{Z}_+(\Delta\setminus \{\alpha\})$, $\lambda \neq \beta$, and $c \in \mathbb{Z}_+$ be such that $\beta + c\alpha \in \Pi_+$ and
%         $\beta - \lambda + c \alpha \in \Pi_+$.
%     Then $\beta - \lambda \in \Pi_+$.
% \end{lemma}

\begin{lemma}\label{l:betalambda}
    Let $\gamma \in \Pi_+$, $\alpha \in \Delta$ and $c>0$ be such that $\gamma - c\alpha \in \Pi_+$. Let $\beta \in\Pi_+$ be such that $\beta \leq \gamma$, $c_{\alpha}(\beta) = c_{\alpha}(\gamma)$ and $\beta \neq \alpha$. Then $\beta-c\alpha \in \Pi_+$.
\end{lemma}
\begin{proof}
    From the given conditions, observe that $\langle \beta,\alpha^{\vee} \rangle \geq \langle \gamma,\alpha^{\vee} \rangle$.
    In the simply-laced case, we see that $c = 1$, and since $\beta \neq \alpha$, the above equation implies $\langle \beta,\alpha^{\vee}\rangle = 1$; hence, $s_{\alpha}(\beta) = \beta - c\alpha \in \Pi_+$.

    For the non-simply-laced case, we consider each type separately. For $G_2$, there are no roots $\gamma$ and $\beta$ that satisfy the conditions of the lemma; for $F_4$, one can observe the root poset directly. Types $B_r$ and $C_r$ are similar in treatment, so let us only verify the type $B_r$.
    
    In type $B_r$, enumerate the simple roots so that $\alpha_r$ is the short one. Since the statement holds for simply-laced types, we assume that $c_{\alpha_r}(\gamma) \neq 0$. If $\gamma = \sum_{k=i}^{j-1}\alpha_k + 2\sum_{k=j}^r \alpha_k$ for $1 \leq i < j \leq r$, then $\alpha \in \{\alpha_i,\alpha_j\}$. If $\alpha = \alpha_i$, then $\beta = \sum_{k=i}^{t-1}\alpha_k + 2\delta \sum_{k=t}^r \alpha_k$ for some $t \in [j,r+1]$ and $\delta \in \{0,1\}$; hence, $\beta-\alpha_i \in \Pi_+$. If $\alpha = \alpha_j$, then $\beta = \sum_{k=p}^{j-1}\alpha_k + 2\sum_{k=j}^r \alpha_k$ for some $p \geq i$, so $\beta-c\alpha_i \in \Pi_+$ (note: $c \in \{1,2\}$ for $j = r$; if $j < r$, then $c = 1$). Thus the statement holds.
\end{proof}

\begin{lemma}\label{l:rat_v_lower}
    Let $u \in W$ be a rational Weyl group element such that $u \neq w_0$. For any simple root $\alpha$ such that $u^{-1}(\alpha) > 0$ and $u(\alpha) < 0$, the element $s_{\alpha}u$ is rational.
\end{lemma}
\begin{proof}
    Assume on the contrary that the element $s_{\alpha}u$ is not rational. Then the rationality graph $\Gamma(u)$ contains a cycle $\{\theta_i \in \nu^0(u)\}_{i \in I}$ where we index the elements by $I:=\mathbb{Z}_m$, so that
    \begin{equation}\label{eq:thetacycle}
        (s_{\alpha}u)^{-1}(\theta_{i-1}) \leq \theta_i, \ i \in I.
    \end{equation}
    Since $u^{-1}(\alpha) > 0$, from equation~\eqref{eq:nu0salpha} we see that
    \begin{equation*}
        \nu^0(s_{\alpha}u) = s_{\alpha}(\nu^0(u)\setminus \{\alpha\});
    \end{equation*}
    therefore, we can write $\theta_i = s_{\alpha}\gamma_i$ for $\gamma_i \in \nu^0(u)\setminus\{\alpha\}$, $i \in I$, and equation~\eqref{eq:thetacycle} becomes
    \begin{equation*}
        u^{-1}(\gamma_{i-1}) \leq s_{\alpha}\gamma_i, \ \ i \in I.
    \end{equation*}
    Define the following subsets of indices:
    \begin{align*}
        &I_-:= \{i \in I \ | \ u^{-1}(\gamma_{i-1}) \leq \gamma_i \};\\
        &I_+ := I \setminus I_-.
    \end{align*}
    Observe that if $s_{\alpha}\gamma_i = \gamma_i - c_i \alpha$ for some $c_i \geq 0$, then $i \in I_-$; hence, if $i \in I_+$, then necessarily $s_{\alpha}\gamma_i = \gamma_i + c_i \alpha$ for some $c_i > 0$. Therefore, there is a number $c_i^\prime$ such that $0 < c_i^\prime \leq c_i$ and $\lambda_i \in \mathbb{Z}(\Delta\setminus\{\alpha\})$ such that 
    \begin{equation*}
        u^{-1}(\gamma_{i-1}) + \lambda_i = \gamma_i + c_i^\prime \alpha.
    \end{equation*}
    By Lemma~\ref{l:betalambda}, since $\gamma_i-\lambda_i + c_i^\prime \alpha \in \Pi_+$ and $\gamma_i + c_i^\prime \alpha \in \Pi_+$, we see that\footnote{To align with the statement of Lemma~\ref{l:betalambda}, $\beta := \gamma_i - \lambda_i + c_i^\prime \alpha$ and $\gamma := \gamma_i + c_i^\prime \alpha$; so, $\gamma - c_i^\prime \alpha \in \Pi_+$. The case $\beta = \alpha$ is excluded due to the assumption $u(\alpha) < 0$ (if $\beta = \alpha$, then $\gamma_{i-1} = u(\alpha) < 0$, which is a contradiction).} $\gamma_i - \lambda_i \in \Pi_+$. Applying $u$ to the above equation, we see that
    \begin{equation*}
        \gamma_{i-1} = u(\gamma_i-\lambda_i) + c_i^\prime u(\alpha),
    \end{equation*}
    and since $u(\alpha) < 0$, we obtain the equation
    \begin{equation}\label{eq:ipluseq}
        \gamma_{i-1} < u(\gamma_i - \lambda_i).
    \end{equation}
    In particular, $u(\gamma_i - \lambda_i) \in \nu^0(u)$. We now exhibit paths in $\Gamma(u)$ which, when combined together, yield a cycle in $\Gamma(u)$.
    \begin{description}
        \item[Case $i,i+1 \in I_-$.] There is an arrow
        \begin{equation*}
            \gamma_{i} \rightarrow \gamma_{i+1}.
        \end{equation*}
        \item[Case $i \in I_-$, $i+1 \in I_+$.] It follows from equation~\eqref{eq:ipluseq} that
        \begin{equation*}
            u^{-1}(\gamma_{i-1}) \leq \gamma_i < u(\gamma_{i+1} - \lambda_{i+1}),
        \end{equation*}
        and therefore, there is a path
        \begin{equation*}
            \gamma_{i-1} \rightarrow u(\gamma_{i+1}-\lambda_{i+1}) \rightarrow \gamma_{i+1}.
        \end{equation*}
        \item[Case $i,i+1 \in I_+$.] Due to equation~\eqref{eq:ipluseq}, there are inequalities
        \begin{equation*}
            u^{-1}(u(\gamma_i-\lambda_i)) = \gamma_i-\lambda_i \leq \gamma_i < u(\gamma_{i+1}-\lambda_{i+1}),
        \end{equation*}
        and therefore, there is a path
        \begin{equation*}
            u(\gamma_i-\lambda_i) \rightarrow u(\gamma_{i+1}-\lambda_{i+1}) \rightarrow \gamma_{i+1}.
        \end{equation*}
        \item[Case $i\in I_+$, $i+1 \in I_-$.] There is an arrow
        \begin{equation*}
            u(\gamma_i-\lambda_i) \rightarrow \gamma_{i+1}.
        \end{equation*}
    \end{description}
    We thus conclude that $s_{\alpha}u$ is rational.
\end{proof}

\begin{proposition}\label{p:left_graph}
    The rationality graph $\Gamma(W)$ is connected.
\end{proposition}

\begin{proof}
    Since $w_0$ is a rational Weyl group element, it suffices to prove there is a path in $\Gamma(W)$ between any rational $u \in W$ and $w_0$. Write $u = v w_0$ for some $v \in W$, and proceed by induction on $\ell(v)$. If $\ell(v) = 0$, then $u = w_0$ is rational. For a fixed $k > 0$, assume there is a path in $\Gamma(W)$ between $u$ with $\ell(v) = k-1$ and $w_0$. By Lemma~\ref{l:correct_simple} and Lemma~\ref{l:rat_v_lower}, there exists a simple root $\alpha$ such that $v^{-1}(\alpha) < 0$ and $s_{\alpha}u$ is a rational element. But then, $\ell(s_{\alpha}v) = k - 1$, and by the assumption, there is a path between $s_{\alpha}u$ and $w_0$; thus, there is a path between $u$ and $w_0$.
\end{proof}

\begin{corollary}\label{c:morethanonevert}
    The rationality graph $\Gamma(W)$ has more than one vertex only in the following Lie types: $A_{r}$ for $r \geq 2$, $D_r$ for $r$ odd, and $E_6$.
\end{corollary}
\begin{proof}
    Indeed, by Proposition~\ref{p:left_graph}, the rationality graph is connected. However, for $\alpha \in \Delta$, $s_{\alpha}w_0$ is rational if and only if $w_0(\alpha) \neq -\alpha$. Indeed, if $w_0(\alpha) = -\alpha$, then $s_{\alpha}w_0(\alpha) = \alpha$, and by Proposition~\ref{p:fixedpt}, $s_{\alpha}w_0$ is not rational. This implies that $w_0$ is connected to a rational Weyl group element in $\Gamma(W)$ if and only if $w_0$ induces a nontrivial automorphism of the corresponding Dynkin diagram. It is so precisely in the Lie types from the statement of the corollary.
\end{proof}

For the next proposition, let $\epsilon : \Delta \rightarrow \Delta$ be an automorphism of the Dynkin diagram such that $w_0(\alpha) = -\epsilon(\alpha)$. Then $\epsilon$ induces an involutive automorphism of the corresponding Weyl group $W$ via $\epsilon(s_{\alpha}):=s_{\epsilon(\alpha)}$.

\begin{proposition}\label{p:zmod2symmetry}
    For any element $u \in W$, $u$ is rational if and only if $\epsilon(u)$ is rational.
\end{proposition}
\begin{proof}
    Equivalently, for $u \in W$, set $u = vw_0$. Since $\epsilon(w_0) = w_0$, we see that $u$ is rational if and only if $\epsilon(v)w_0$ is rational. Now the result follows from the observation that the elements of $\nu^0(u)$ can be determined via a reduced decomposition of $v$ (see Remark~\ref{r:howtonu0}).
\end{proof}

\subsection{Proof of Proposition~\ref{p:rat_coxeter} (Coxeter elements)}\label{s:rat_coxeter}

In this subsection, we prove Proposition~\ref{p:rat_coxeter}, which states that the only rational Coxeter elements are in type $A_r$, and they are of the form $C:=s_1\cdots s_r$ and $C^{-1}$. We split the proposition into three separate statement: for type $A_r$, Proposition~\ref{p:onlycoxeterA} and Proposition~\ref{p:coxeter}; for type $D_r$, Proposition~\ref{p:coxeterD}; for type $E_6$, the statement can be verified with  computer software (see Proposition~\ref{p:coxeterE6}). For all the other types, the statement is a consequence of Corollary~\ref{c:morethanonevert}.

For the next results, we model the root system of type $A_r$ in an $(r+1)$-dimensional Euclidean space $E^{r+1}$ endowed with an orthonormal basis $e_1,\ldots,e_{r+1}$. The set of positive roots is given by
\begin{equation*}
    \Pi_+ = \{e_{i} - e_{j} \ | \ 1 \leq i < j \leq r+1\},
\end{equation*}
and the set of simple roots is given by
\begin{equation*}
    \Delta = \{\alpha_i:=e_{i} - e_{i+1} \ | \ 1 \leq i \leq r\}.
\end{equation*}
Each simple reflection $s_i$ acts on $E^{r+1}$ by permuting the vectors $e_{i}$ and $e_{i+1}$ and fixing $e_j$ for $j \notin \{i,i+1\}$. The Coxeter element $C$ acts by
\begin{equation*}
    C(e_i) = \begin{cases}
        e_{i+1} &i \in \{1,\ldots,r\}\\
        e_1 &i = {r+1}.
    \end{cases}
\end{equation*}
For the next lemma, set 
\begin{align*}
    &C_1 := s_1 s_2 \cdots s_r = C;\\
    &C_2 := s_2 s_3 \cdots s_r s_1;\\
    &\vdots\\
    &C_m := s_{m} s_{m+1}\cdots s_rs_{m-1}s_{m-2}\cdots s_1;\\
    &\vdots\\
    &C_{r-1} := s_{r-1} s_r s_{r-2} s_{r-3} \cdots s_{1}\\
    &C_r := s_r s_{r-1} \cdots s_1 = C^{-1}.
\end{align*}

% \begin{equation}
%     C_m := s_{m} s_{m+1}\cdots s_rs_{m-1}s_{m-2}\cdots s_1, \ \ m \in \{1,\ldots,r\}.
% \end{equation}
% In particular, $C_1 = C$ and $C_r = C^{-1}$.

\begin{lemma}\label{l:coxelhighroot}
Let $c$ be a Coxeter element in type $A_r$ and $\theta$ be the highest root. Then $c(\theta) > 0$ if and only if $c \notin \{C_1,C_2,\ldots,C_r\}$. Moreover, $C_m(\theta) = -\alpha_m$, $m \in \{1,\ldots,r\}$.
\end{lemma}
\begin{proof}
    Note that the highest root $\theta$ is given by $\theta = e_1 - e_{r+1}$, and a direct check shows that $C_m(\theta) = -\alpha_m$. Let us embed the Euclidean space $E^r$ into $E^{r+1}$ as the orthogonal complement of the vector $e_{r+1}$. We proceed by induction on $r$, with cases $r = 1$ and $r=2$ being evident. A Coxeter element $c$ can be written as $s_{i_1}\cdots s_{i_{r-1}}s_r$ or as $c = s_r s_{i_1} \cdots s_{i_{r-1}}$. In the first case, we see that 
    \begin{equation*}
        s_{i_1}\cdots s_{i_{r-1}}s_r(e_1 - e_{r+1}) = s_{i_1} \cdots s_{i_{r-1}}(e_1 - e_r),
    \end{equation*}
    and hence the statement follows from the case of $A_{r-1}$. In the second case, we see that  $c(\theta) = s_r(e_i - e_{r+1})$ for some $i \in \{1,\ldots,r\}$. If $i < r$, then $c(\theta) > 0$; if $i = r$, then $c(\theta) = -\alpha_r$, and in this case, $c = s_r s_{r-1} \cdots s_1$.
\end{proof}

\begin{proposition}\label{p:onlycoxeterA}
    In type $A_{r}$, a Coxeter element $c \in W$ is rational if and only if $c \in \{C,C^{-1}\}$.
\end{proposition}
\begin{proof}
    By Lemma~\ref{l:coxelhighroot}, we see that if $c \notin \{C_1^{-1},C_2^{-1},\ldots,C_r^{-1}\}$, then $c^{-1}(\theta) > 0$, and therefore, $\theta \in \nu^0(c)$; by Proposition~\ref{p:highrootnonrat}, $c$ is not rational.

    Assume that $c = C_m^{-1}$ for some $m \in \{2,\ldots,r-1\}$. Observe that there is a $2$-cycle in $\Gamma(c)$:
    \begin{align*}
        &(c^{-1})(e_1 - e_{m+1}) = e_{m+1}-e_{m+2} \leq e_m - e_{r+1};\\
        &(c^{-1})(e_m - e_{r+1}) = e_{m-1}-e_m \leq e_1 - e_{m+1};
    \end{align*}
    hence, $c$ is not rational.
    
    It remains to show that if $c \in \{C,C^{-1}\}$, then $c$ is rational. Observe that
    \begin{equation*}
        \nu^{k-1}(C) = \{e_i - e_j \ | \ k < i < j \leq r+1\}, \ \ k \geq 1.
    \end{equation*}
    Evidently, $\nu(C) = \emptyset$. To show that $C^{-1}$ is rational, observe that $\epsilon(C) = C^{-1}$ (the $\mathbb{Z}_2$-symmetry of $\Gamma(W)$, see Proposition~\ref{p:zmod2symmetry}), hence $C^{-1}$ is rational.
\end{proof}

\begin{proposition}\label{p:coxeter}
    In type $A_{r}$, the Coxeter elements $C$ and $C^{-1}$ are rational and have valency $1$ in the graph $\Gamma(W)$. More precisely, $C$ is connected only to the rational element $s_{r}C$, and $C^{-1}$ is connected only to the rational element $s_1 C^{-1}$.
\end{proposition}
\begin{proof}
    The Coxeter elements $C$ and $C^{-1}$ are rational by Proposition~\ref{p:onlycoxeterA}. We see that $C(\alpha_{r}) < 0$ and $C^{-1}(\alpha_{r}) > 0$; hence, by Lemma~\ref{l:rat_v_lower}, the element $s_{r}C$ is rational. Likewise, by the same lemma, the element $s_1 C^{-1}$ is rational.

   Let us show that the valency of $C$ in $\Gamma(W)$ is equal to $1$. Consider the element $s_i C$ for $i \in \{1,\ldots,r-1\}$. We see that 
    \begin{equation*}
        s_i C(\alpha_i) = s_i(\alpha_{i+1}) = \alpha_i + \alpha_{i+1} \geq \alpha_i;
    \end{equation*}
    therefore, by Proposition~\ref{p:fixedpt}, $s_i C$ is not rational, and thus the valency of $C$ is  $1$. By the $\mathbb{Z}_2$-symmetry of $\Gamma(W)$ (see Proposition~\ref{p:zmod2symmetry}), the valency of $C^{-1}$ is also $1$.
\end{proof}

For the next results, we model the root system of type $D_r$ in an $r$-dimensional Euclidean space $E^r$ endowed with an orthonormal basis $e_1,\ldots,e_r$. Then the set of positive roots is given by
\begin{equation*}
    \Pi_+ = \{e_i - e_j \ | \ 1 \leq i < j \leq r\} \cup \{e_i + e_j \ | \ 1 \leq i < j \leq r\},
\end{equation*}
and the set of simple roots is given by
\begin{equation*}
    \Delta = \{\alpha_i:=e_i -e_{i+1} \ | \ 1 \leq i \leq r-1\}\cup\{\alpha_r:=e_{r-1} + e_r\}.
\end{equation*}
In this model, each simple reflection $s_i$ for $i \in \{1,\ldots,r-1\}$ interchanges $e_i$ and $e_{i+1}$, and leaves the other basis vectors fixed; for $s_r$, $s_r(e_{r-1}) = -e_r$.

\begin{lemma}\label{l:coxelhighrootD}
    Let $\theta$ be the highest root vector in type $D_r$. For any Coxeter element $c$, $c(\theta) > 0$.
\end{lemma}
\begin{proof}
    Note that the highest root $\theta$ in type $D_r$, in the given Euclidean model, is $\theta = e_1 + e_2$.
    Consider two cases based on the position of the simple reflection $s_r$ in a reduced decomposition of~$c$.
    \begin{description}
        \item[Case 1.] Assume that $c = s_{i_1}s_{i_2}\cdots s_{i_{r-1}}s_r$. Since $s_r(\theta) = \theta$, and since no other simple reflection can change the sign of $\theta$, we see that $c(\theta) > 0$.
        \item[Case 2.] Assume that $c = s_r s_{i_1}\cdots s_{i_{r-1}}$. Then $s_{i_1}\cdots s_{i_{r-1}}(\theta) = e_i + e_j$ for some $i \neq j$. Then $c(\theta) < 0$ if and only if $i = r-1$ and $j = r$. However, $s_{i_1}\cdots s_{i_{r-1}}(e_1) = e_{r-1}$ if and only if $(i_1,\ldots,i_{r-2},i_{r-1}) = (r-1,r-2,\ldots,1)$, but in this case, $s_{r-1}s_{r-2}\cdots s_1(e_2) = e_2$.
    \end{description}
    Thus $c(\theta) > 0$ for any Coxeter element $c$ in type $D_r$.
\end{proof}

\begin{proposition}\label{p:coxeterD}
    In type $D_r$, no Coxeter element is rational.
\end{proposition}
\begin{proof}
    The statement is a direct consequence of Lemma~\ref{l:coxelhighrootD} and Proposition~\ref{p:highrootnonrat}.
\end{proof}

\begin{proposition}\label{p:coxeterE6}
In type $E_6$, no Coxeter element is rational.
\end{proposition}
\begin{proof}
    This can be verified with computer software via seeing, as before, that for every Coxeter element $c$ in type $E_6$ and the highest root $\theta$, $\theta \in \nu^0(c)$.
\end{proof}

\subsection{\texorpdfstring{Proof of Proposition~\ref{p:typeDrvalency1} (valency $1$ vertices in $\Gamma(D_r)$)}{Proof of Proposition~\ref{p:typeDrvalency1}}}\label{s:typeDrvalency1}
In this subsection, we construct a pair of rational elements in type $D_r$ of valency $1$ in the rationality graph $\Gamma(D_r)$. We establish some of its basic properties in Proposition~\ref{p:specialDr}, and we prove Proposition~\ref{p:typeDrvalency1} (whose more precise statement is Proposition~\ref{p:dcox}).

We represent the root system $D_r$ in a Euclidean space $E^r$ in the same way as in Section~\ref{s:rat_coxeter}. Define the element $C$ by its action upon the orthonormal basis $e_1,\ldots,e_r$, as follows:
\begin{equation*}
    C(e_i) = \begin{cases}
        -e_{r} & \text{if} \ i = 1\\
        -e_i & \text{if} \ i \in \{2,3,\ldots,r-1\}\\
        e_1 &\text{if} \ i = r.
    \end{cases}
\end{equation*}
%double-checked on Matlab
%Its eigenvalues belong to the set $\{-1,i,-i\}$.

\begin{proposition}\label{p:specialDr}
    The element $C$ belongs to $W$; moreover, its length is equal to 
    \begin{equation*}
        \ell(C) = \frac{1}{2}(r(r-1) + (r-2)(r-3)),
    \end{equation*}
    %double-checked on Matlab
    its order is equal to $4$, and its reduced decomposition can be constructed as\footnote{The square parentheses carry no extra meaning and are to be interpreted as round parentheses.}
    \begin{equation}\label{eq:cdr}
    \begin{split}
        C = &s_r (s_{r-2}s_{r-1})[s_r] (s_{r-3}s_{r-2}s_{r-1}) [s_r s_{r-2}] (s_{r-4} s_{r-3} s_{r-2} s_{r-1})\cdot \\ \cdot &[s_r s_{r-2}s_{r-3}] (s_{r-5}s_{r-3} s_{r-2}s_{r-1})\cdots [s_{r}s_{r-2} s_{r-3} \cdots s_3](s_1 s_2 \cdots s_{r-2}s_{r-1}).
        \end{split}
    \end{equation}
    Lastly, $\epsilon(C) = C^{-1}$ where $\epsilon$ is the $\mathbb{Z}_2$-symmetry of $\Gamma(D_r)$.
\end{proposition}
%indices in Matlab for C in D7: 
%      307774
%      309127
%double-checked the formula for C on Matlab

\begin{proof}
    The formulas can be verified via a direct computation in the orthonormal basis $e_1,\ldots,e_r$. To verify that $C \cdot \epsilon(C) = \id$, one computes directly the product using the decomposition~\eqref{eq:cdr}.
\end{proof}

\begin{proposition}\label{p:dcox}
    In type $D_r$ for $r \geq 5$ odd, the elements $C$ and $C^{-1}$ are rational and have valency~$1$ in the rationality graph $\Gamma(D_r)$. More precisely, $C$ is connected only to the rational element $s_{r-1}C$, and $C^{-1}$ is connected only to the rational element $s_rC^{-1}$.
\end{proposition}
\begin{proof}
    From the $\nu$-sequence for $C$, we see that
    \begin{align*}
        &\nu^0(C) = \{e_1 - e_i \ | \ i \in \{2,\ldots,r-1\} \} \cup \{e_i-e_r \ | \ i \in \{1,\ldots,r-1\}\};\\
        &\nu^1(C) = \{e_i-e_r \ | \ i \in \{2,\ldots,r-1\}\};\\
        &\nu^2(C) = \emptyset;
    \end{align*}
    therefore, $C$ is rational. By the $\mathbb{Z}_2$-symmetry of $\Gamma(D_r)$ (see Proposition~\ref{p:zmod2symmetry}), $C^{-1}$ is also rational. We see that
    \begin{equation*}
        C(e_{r-1}-e_r) = -(e_1 + e_{r-1}) < 0, \ \ C^{-1}(e_{r-1}-e_r) = e_1 -e_{r-1} > 0.
    \end{equation*}
    By Lemma~\ref{l:rat_v_lower}, $s_{r-1}C$ is rational. Likewise, one verifies that $s_rC^{-1}$ is rational by the same lemma.     
    
    Now, let us show that $s_i C$ for $i \in \{1,\ldots,r-2\}\cup\{r\}$ is not rational. For $i \in \{1,\ldots,r-2\}$, we see that
    \begin{equation*}
        s_iC(e_i-e_{i+1}) = e_i - e_{i+1};
    \end{equation*}
    and for $i = r$, we see that
    \begin{equation*}
        s_r C (e_{r-1} + e_r) = s_r (e_1 - e_{r-1}) = e_1 + e_r \geq e_{r-1} + e_r.
    \end{equation*}
    By Proposition~\ref{p:fixedpt}, $s_iC$ for $i\in \{1,\ldots,r-2\}\cup\{r\}$ is not rational.
\end{proof}
   
  \newpage
\appendix
\section{\texorpdfstring{Rationality graphs $\Gamma(A_3)$, $\Gamma(A_4)$ and $\Gamma(D_5)$}{Rationality graphs in types A3, A4 and D5}}\label{appendix}
  %In the appendix, we provide illustrations of rationality graphs in type $A_3$, $A_4$ and $D_5$. 

\begin{figure}[htb]
\begin{subfigure}[t]{6.5in}
\begin{center}
\includegraphics[scale=0.2]{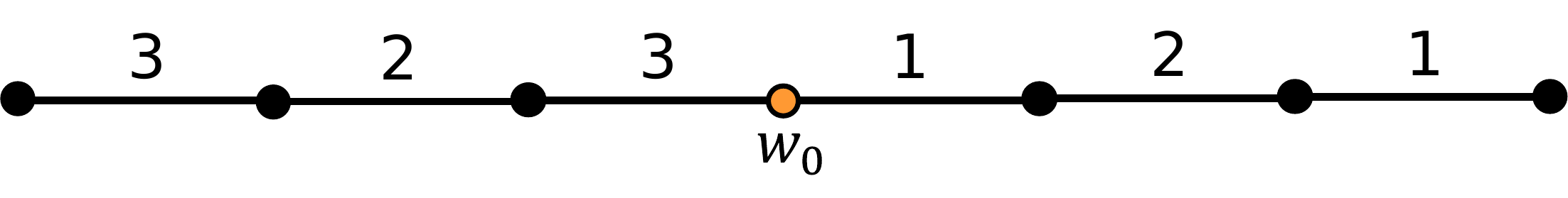}
\end{center}
\caption{Rationality graph $\Gamma(A_3)$.}
\label{f:a3}
\end{subfigure}

\vspace{10mm}

\begin{subfigure}[t]{6.5in}
\begin{center}
\includegraphics[scale=0.2]{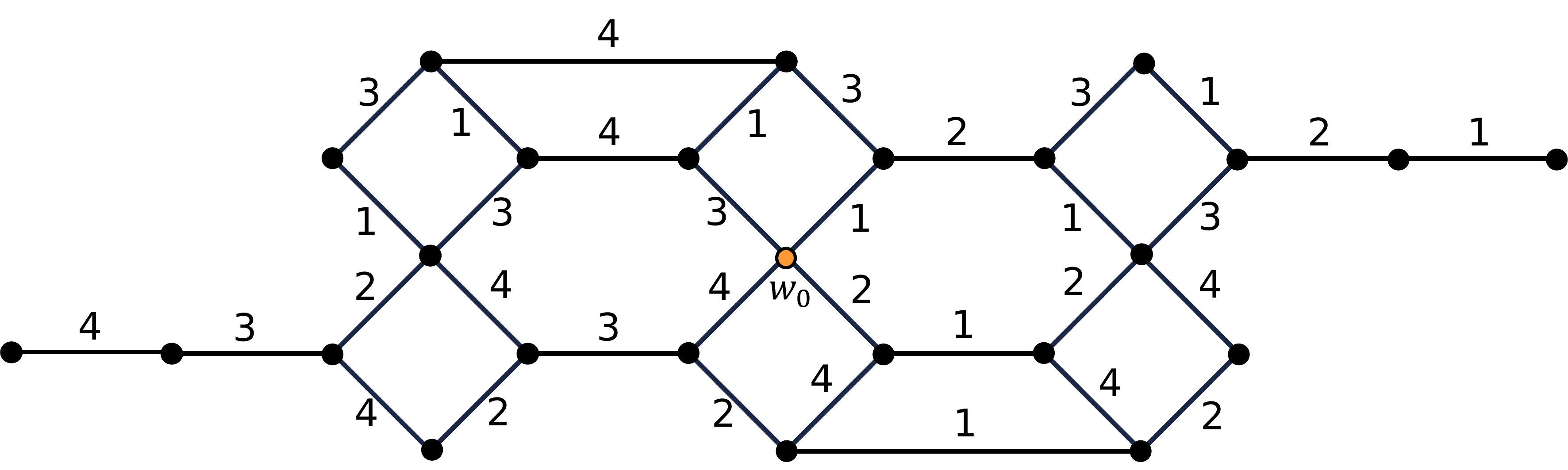}
\end{center}
\caption{Rationality graph $\Gamma(A_4)$.}
\label{f:a4}
\end{subfigure}

\vspace{10mm}

\begin{subfigure}[t]{6.5in}
\begin{center}
\includegraphics[scale=0.2]{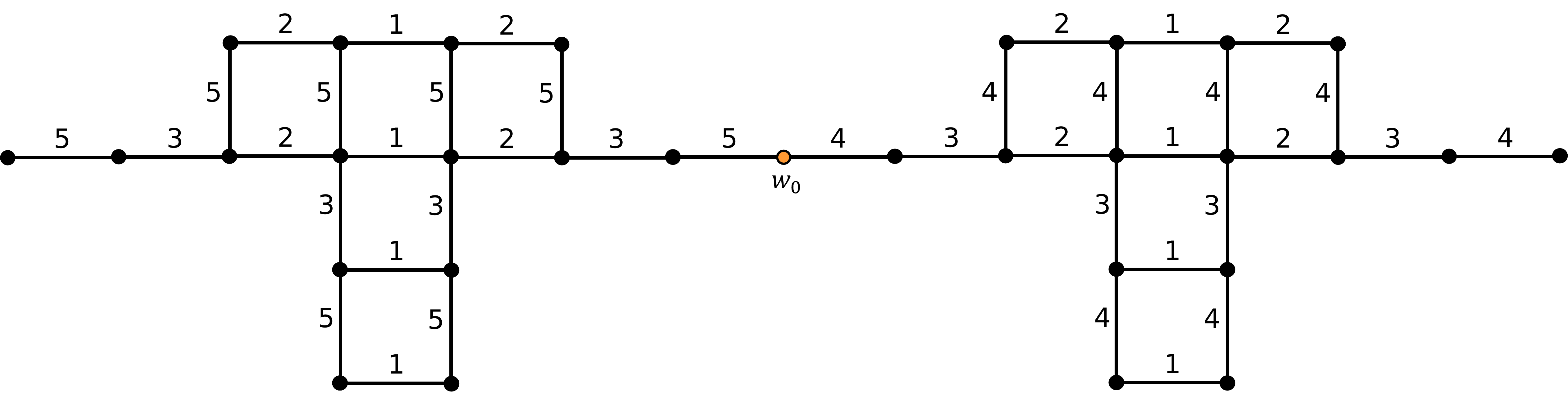}
\end{center}
\caption{Rationality graph $\Gamma(D_5)$.}
\label{f:d5}
\end{subfigure}
\caption{Rationality graphs in types $A_3$, $A_4$ and $D_5$. Each label $i$ of an edge is to be interpreted as a simple reflection $s_i$.}
\label{f:ratgraphs}
\end{figure}


\begin{thebibliography}{99}

% \bibitem{bourbaki4-7} N. Bourbaki, \emph{Lie Groups and Lie Algebras. Chapters 4-6}

\bibitem{armstrong} D. Armstrong, Generalized noncrossing partitions and combinatorics of Coxeter groups, \emph{Mem. Amer. Math. Soc.} \textbf{202}(949) (2009), 1--159. \doi{10.1090/S0065-9266-09-00565-1}

\bibitem{schubertpos} A. Berenstein and A. Zelevinsky, Total positivity in Schubert varieties, \emph{Comment. Math. Helv.} \textbf{72} (1997), 128--166. \doi{10.1007/PL00000363}

\bibitem{chari} V. Chari and A. Pressley, \emph{A Guide to Quantum Groups}, Cambridge Univ. Press (1995). 

\bibitem{etingof} P. Etingof and O. Schiffmann, \emph{Lectures on Quantum Groups} 1st ed., International Press, Somerville, MA (1998).

 \bibitem{fg-moduli} V. Fock and A. Goncharov, Moduli spaces of local systems and higher Teichm\"uller theory, \emph{Publ. Math. Inst. Hautes Etudes Sci.} \textbf{103}(2006), 1--211. \doi{10.1007/s10240-006-0039-4}
 
\bibitem{fomin13} S. Fomin, L. Williams and A. Zelevinsky, Introduction to Cluster algebras. Chapters 1-3, Preprint (2020). \doi{10.48550/arXiv.1608.05735}
%\href{https://arxiv.org/abs/1608.05735}{arXiv:1608.05735}

 \bibitem{fathers} S. Fomin and A. Zelevinsky, Cluster algebras I: Foundations, \emph{J. Amer. Math. Soc.} \textbf{15}(2) (2002), 497--529. \doi{10.1090/S0894-0347-01-00385-X}

\bibitem{bruhatpos} S. Fomin and A. Zelevinsky, Double Bruhat cells and total positivity, \emph{J. Amer. Math. Soc.} \textbf{12}(2) (1999), 335--380.
\doi{10.1090/S0894-0347-99-00295-7}

\bibitem{fomin_intro} S. Fomin, Total positivity and cluster algebras, In: \emph{Proc. Int. Cong. Math.} Vol. II, 125--145, Hindustan Book Agency, New Delhi (2010). \doi{10.1142/9789814324359_0043}

 \bibitem{dasbuch} M. Gekhtman, M. Shapiro and A. Vainshtein, {C}luster algebras and {P}oisson geometry, \emph{Math. Surveys Monogr.} \textbf{167} (2010). \doi{10.1090/surv/167}
 
 \bibitem{double} M. Gekhtman, M. Shapiro and A. Vainshtein, Drinfeld double of $\GL_n$ and generalized cluster structures, \emph{Proc. Lond. Math. Soc} \textbf{116}(3) (2018), 429--484. \doi{10.1112/plms.12086}

 \bibitem{multdual} M. Gekhtman and D. Voloshyn, Generalized cluster structures related to Poisson duals of $\SL_n$, Preprint (2023).  \doi{10.48550/arXiv.2312.04859}

%\bibitem{harris} J. Harris, \emph{Algebraic geometry: a first course.} Springer Science \& Business Media Vol. 133 (1992). \doi{10.1007/978-1-4757-2189-8}

\bibitem{humphreys_r} J. E. Humphreys, \emph{Reflection groups and Coxeter groups}, Cambridge Univ. Press (1990).
\doi{10.1017/CBO9780511623646}

 \bibitem{leclerc} B. Leclerc, Cluster algebras and representation theory, In: \emph{Proc. Int. Cong. Math.} Vol. IV, 2471--2488, Hindustan Book Agency, New Delhi (2010). \doi{10.1142/9789814324359_0154}

 \bibitem{multdouble} D. Voloshyn, Multiple generalized cluster structures on $D(\mathrm{GL}_n)$, \emph{Forum of Mathematics, Sigma} (\textbf{11})(46) (2023), 1--78. \doi{10.1017/fms.2023.44}

 \bibitem{williams_intro} L.K. Williams, Cluster algebras: an introduction, \emph{Bull. Amer. Math. Soc. (N.S.)} \textbf{51} (2014), 1--26. \doi{10.1090/S0273-0979-2013-01417-4}

  \end{thebibliography}
 \end{document}